\documentclass[11pt,reqno]{article}
\usepackage[titletoc,toc,title]{appendix}
\usepackage[T1]{fontenc} 
\usepackage[utf8]{inputenc}
\usepackage{booktabs}
\usepackage{array}
\usepackage{paralist}
\usepackage{verbatim}
\usepackage{amsthm}
\usepackage[table,xcdraw]{xcolor}
\usepackage{latexsym, amsmath, amssymb, a4, epsfig, color}
\usepackage{blindtext}
\usepackage{graphicx}
\usepackage[export]{adjustbox}
\usepackage{wrapfig}
\usepackage{apptools}
\usepackage{afterpage}
\usepackage{floatrow}
\usepackage[normalem]{ulem}
\usepackage{appendix}
\usepackage{booktabs} % for much better looking tables
\usepackage{array} % for better arrays (eg matrices) in maths
\usepackage{paralist}
\usepackage{verbatim} % adds environment for commenting out blocks of text & for better verbatim
\usepackage{amsthm}
\usepackage[table,xcdraw]{xcolor}
\usepackage{multirow}
\usepackage{array}
\newcolumntype{C}[1]{>{\centering\arraybackslash$}p{#1}<{$}}

\usepackage[french,english]{babel}

\AtAppendix{\counterwithin{definition}{subsection}}
\AtAppendix{\counterwithin{theorem}{subsection}}

\graphicspath{}

\newtheorem{theorem}{Theorem}[section]

\newtheorem{lemma}[theorem]{Lemma}
\newtheorem{prop}[theorem]{Proposition}
\newtheorem{definition}{Definition}

\newtheorem{example}{Example}
\newtheorem{remark}{Remark}

\usepackage{subfig}
\usepackage{lipsum}
\usepackage{amsfonts}

\usepackage{epstopdf}
\usepackage{algorithmic}
\ifpdf
  \DeclareGraphicsExtensions{.eps,.pdf,.png,.jpg}
\else
  \DeclareGraphicsExtensions{.eps}
\fi

\usepackage{enumitem}
\setlist[enumerate]{leftmargin=.5in}
\setlist[itemize]{leftmargin=.5in}

\newcommand{\FF}{\mathbb{F}}
\newcommand{\RR}{\mathbb{R}}
\newcommand{\CC}{\mathbb{C}}
\newcommand{\NN}{\mathbb{N}}
\newcommand{\NNc}{\mathcal{N}}
\newcommand{\MM}{\mathbb{M}}

\newcommand{\Om}{\Omega}
\newcommand{\ds}{\displaystyle}

\newcommand{\p}{\partial}
\newcommand{\pd}[2]{\frac {\p #1}{\p #2}}
\newcommand{\eqnref}[1]{(\ref {#1})}
\renewcommand{\qed}{\hfill $\Box$ \medskip}
\newcommand{\beq}{\begin{equation}}
\newcommand{\eeq}{\end{equation}}
\newcommand{\be}{\begin{equation*}}
\newcommand{\ee}{\end{equation*}}
\newcommand{\ba}{\begin{align*}}
\newcommand{\ea}{\end{align*}}
\newcommand{\bal}{\begin{align}}
\newcommand{\eal}{\end{align}}

\def\ep{\varepsilon}

\def\ep{\varepsilon}

\newcommand{\Kcal}{{K}}

\newcommand{\la}{\langle}
\newcommand{\ra}{\rangle}

\newcommand{\arcsinh}{\operatorname{arcsinh}}

\title{Inverse problem for a planar conductivity inclusion\thanks{This study was supported by National Research Foundation of Korea (NRF) grant funded by the Korean government (MSIT) (NRF-2021R1A2C1011804) and by the Swedish Research Council under contract 2021-03720.}}
\author{Doosung Choi
\thanks{Department of Mathematics, Louisiana State University, Baton Rouge, LA 70808, USA.}$\,\,$\footnotemark[4]
\and
Johan Helsing\thanks{Centre for Mathematical Sciences, Lund University, Box 118, 221 00 Lund, Sweden ({johan.helsing@math.lth.se}).}
\and Sangwoo Kang\thanks{Department of Mathematical Sciences, Korea Advanced Institute of Science and Technology, Daejeon 34141, Republic of Korea ({7john@kaist.ac.kr}, {sangwoo.kang@kaist.ac.kr},  {mklim@kaist.ac.kr}).}
\and Mikyoung Lim\footnotemark[4]}
\date{\today}

\begin{document}

\maketitle

\begin{abstract}
This paper concerns the inverse problem of determining a planar conductivity inclusion. Our aim is to analytically recover from the generalized polarization tensors (GPTs), which can be obtained from exterior measurements, a homogeneous inclusion with arbitrary constant conductivity. The primary outcome of recovering a homogeneous inclusion is an inversion formula in terms of the GPTs for conformal mapping coefficients associated with the inclusion. To prove the formula, we establish matrix factorizations for the GPTs.
\end{abstract}



\section{Introduction}

The problem of determining electrical conductivity throughout a domain from boundary field measurements is of great interest which goes back many years \cite{Calderon:1980:IBV,Cannon:1967:DUC,Tikhonov:1949:USP,Tikhonov:1950:DEC}.
It has been extensively studied given its importance in real-life applications such as medical imaging and nondestructive testing. 
For instance, we refer to \cite{Brown:1997:UIC,Caro:2016:GUC,Druskin:1982:USI,Druskin:1998:UIP,Nachman:1996:GUT,Sylvester:1987:GUT} for the uniqueness results and  to \cite{Alessandrini:1990:SSE,Greenleaf:2003:CPC,Isaacson:2004:RCP,Knudsen:2004:RLR, Kohn:1984:DCB,Kohn:1985:DCB,Novikov:1988:MIS,Novikov:2009:EGR,Paivarinta:2003:CGO} for reconstruction methods. We also refer to \cite{Adler:2021:EIT,Borcea:2002:EIT,Uhlmann:2009:EIT} and the references therein for more results.
Despite the theoretical and numerical results achieved, developing analytic inversion formulas is challenging because of the consequential nature of the nonlinearity and complexity of the inverse problem.

As building blocks for the detection problem of a conductivity inclusion, one can use the so-called generalized polarization tensors (GPTs), which are complex-valued matrices that generalize the polarization tensor (PT) \cite{Ammari:2013:MSM, Polya:1951:IIM}. 
More precisely,  the GPTs are the coefficients of the multipole expansion for the potential function that is perturbed due to the existence of the inclusion. They can be obtained from multistatic measurements \cite{Ammari:2007:PMT}, where a high signal-to-noise ratio is required for high-order terms \cite{Ammari:2014:TIU}. 
Efficient algorithms have been developed to determine the location and shape of inhomogeneities from the GPTs \cite{Ammari:2013:TMT,Ammari:2018:MNF,Ammari:2013:MSM,Ammari:2014:GPT,Ammari:2019:RDA,Baldassari:2021:MCE,Choi:2018:CEP,Choi:2021:ASR} (see also \cite{Ammari:2003:PGP} for the uniqueness result and \cite{Ammari:2018:MNF,Choi:2018:GME,Feng:2018:CGV} for other applications).

For the case of a planar simply connected inclusion, which is the focus of this paper, analytic shape recovery algorithms have been developed based on the complex analytical formulation for the conductivity transmission problem. The conformal mapping coefficients were explicitly expressed by the GPTs under the assumption that the inclusion is perfectly conducting or insulating \cite{Choi:2018:CEP,Choi:2021:ASR,Kang:2015:CCM}.  However, only optimization approaches have been developed for an inclusion with general conductivity \cite{Ammari:2012:GPT, Choi:2021:ASR}. The instances with arbitrary finite conductivity pose a specific complication because, unlike the inclusion with extreme conductivity, neither the Dirichlet nor the Neumann boundary condition is explicitly given in advance. It remains an interesting and open problem to generalize the inversion formula for the conformal mapping coefficients of an inclusion with extreme conductivity to the case of the inclusion with arbitrary finite conductivity. The objective of this paper is to provide a solution to this problem.

As the main tool, we use the concept of the Faber polynomial polarization tensors (FPTs), that is, the linear combinations of the GPTs with expansion coefficients defined by the Faber polynomials \cite{Choi:2018:GME}. 
Indeed, for any simply connected domain in the complex plane, the Riemann mapping theorem assures the existence and uniqueness of the conformal mapping that transforms the exterior of a disk to the exterior of the domain.
Then, this exterior conformal mapping generates the Faber polynomials, which form a basis for complex analytic functions in the domain. Recently, the FPTs were successfully applied to the asymptotic shape recovery of a conductivity inclusion \cite{Choi:2021:ASR} and the properties of the PT \cite{Cherkaev:2021:GSE}. It is worth remarking that the layer potential operators associated with the domain admit matrix expressions with entries given by the Grunsky coefficients, which are expansion coefficients of the composition of the Faber polynomials and the exterior conformal mapping \cite{Jung:2019:DEE,Jung:2021:SEL}.

As our main contribution, we propose a new factorization method for recovering a planar conductivity inclusion with arbitrary constant conductivity  from the GPTs. For two semi-infinite matrices whose entries are scalar-valued complex contracted GPTs, we derive matrix factorization formulas in terms of the material parameter and conformal mapping coefficients associated with the inclusion. We rigorously prove that the formulas hold for either a smooth domain or a star-shaped domain with a Lipschitz boundary.
Then, through the cancellation of the common factors in the two matrix factorizations, we derive an explicit inversion formula for the conformal mapping coefficients in terms of the GPTs and the conductivity value of the inclusion. We also obtain a fixed-point equation from which one can numerically compute the conductivity value. In conclusion, one can analytically recover the shape of an inclusion that possibly has a Lipschitz boundary after determining the conductivity value by a fixed-point computation. 
Our approach generalizes the shape recovery formulas for an inclusion with extreme conductivity obtained in \cite{Choi:2018:CEP, Choi:2021:ASR, Kang:2015:CCM} to an inclusion with arbitrary finite conductivity. Also, it significantly improves the asymptotic results of \cite{Ammari:2014:GPT, Ammari:2012:GPT, Ammari:2018:IAD,Choi:2021:ASR} in that the resulting inversion formula of the shape recovery holds exactly rather than approximately.

We validate the proposed reconstruction approach with numerical experiments for inclusions of various shapes. To compute the GPTs, we solve a boundary integral equation involving the Neumann--Poincar\'{e} operator by using the Nystr\"{o}m discretization. For domains with corners in the numerical examples, we employ recursively compressed inverse preconditioning (RCIP) to compute the GPTs to a high degree of precision \cite{Helsing:2008:CSE}.

This paper is organized as follows. Section 2 describes the inverse problem of reconstructing a conductivity inclusion from exterior measurements and the concepts of the GPTs and FPTs. Section 3 is devoted to reviewing shape recovery methods in the previous literature. In Section 4, we establish matrix factorizations for the GPTs and derive an inversion formula. In Section 5, we extend the proposed approach to a Lipschitz domain. We then validate our method with numerical examples in Section 6. 
We conclude with Section 7.

%%%%%%%%%%%%%%%%%%%%%%%%%%%%%%%%%%%%%%%%%%%%%%%%%%%%%%%%

\section{Preliminary}
\subsection{Problem formulation}
Let $\Om$ be a simply connected, bounded domain with a Lipschitz boundary in $\RR^2$. We assume that $\p \Om$ is a Jordan curve. We further assume that $\RR^2\setminus\overline{\Om}$ and $\Om$ have constant isotropic conductivities, respectively denoted by $\sigma_m$ and $\sigma_c$, satisfying $0<\sigma_c \neq \sigma_m < \infty$.  Set $\lambda = \frac{\sigma_c + \sigma_m}{2(\sigma_c - \sigma_m)}$ if not specified otherwise. Note that that $|\lambda|>\frac{1}{2}$. 
Consider the conductivity transmission problem:
\beq\label{transmission}
\begin{cases}
\ds \Delta u=0\qquad&\mbox{in } \RR^2 \setminus \p \Om, \\
\ds u\big|^+ = u\big|^- \qquad&\mbox{on }\p \Om,\\
\ds \sigma_m \pd{u}{\nu}\Big|^+ = \sigma_c  \pd{u}{\nu}\Big|^- \qquad&\mbox{on }\p \Om,\\
\ds (u-H)(x)  =O({|x|^{-1}})\qquad&\mbox{as } |x| \to \infty,
\end{cases}
\eeq
where $H$ is a given background potential that is entire harmonic. If there were no inclusion, the solution $u$ would be $H$.    The perturbation $u-H$ due to the inclusion depends on the geometry and material property of the inclusion and can be expressed in terms of layer potentials.

The Neumann--Poincar\'{e} (NP) operator for $\varphi\in L^2(\p \Om)$ is defined as
\begin{align*}
\ds &\Kcal_{\p\Om}^*[\varphi](x)=p.v.\,\frac{1}{2\pi}\int_{\partial \Om}\frac{\left\la x-y,\nu_x\right\ra}{|x-y|^2}\varphi(y)\, d\sigma(y),\quad x\in\p \Om,
\end{align*}
where $p.v.$ stands for the Cauchy principal value and $\nu_x$ is the outward unit normal vector to $\p \Om$ at $x$. The operator $\lambda I-\Kcal_{\p \Om}^*$ is invertible on $L^2_0(\p \Om)$ (or $H_0^{-1/2}(\p \Om)$) for $|\lambda|\geq 1/2$ (see \cite{Escauriaza:1992:RTW, Kellogg:1929:FPT, Verchota:1984:LPR}). 

The solution $u$ admits the multipole expansion \cite{Ammari:2007:PMT}: for $|x|>\sup\left\{|y|:y \in \Omega \right\}$, 
\beq\label{scattered2}
(u - H)(x) =  \sum_{|\alpha|, |\beta| \geq 1} \frac{(-1)^{|\beta|}}{\alpha! \beta!} \partial^\alpha H(0) M_{\alpha \beta}(\Omega, \lambda) \partial^\beta \Gamma(x)
\eeq
with two-dimensional multi-indices $\alpha,\beta$ and the so-called generalized polarization tensors (GPTs)
$$M_{\alpha\beta }(\Omega, \lambda)  =  \int_{\partial \Omega} y^\beta\,(\lambda I - \Kcal_\Omega^*)^{-1} \left[\nu\cdot\nabla x^\alpha\right](y)\, d\sigma(y).$$
Here, $\Gamma(x)$ is the fundamental solution to the Laplacian, i.e., 
$\Gamma(x)=\frac{1}{2\pi}\ln|x|$.
We refer the reader to  \cite{Ammari:2003:PGP} for the uniqueness of the inverse problem of determining the shape and conductivity value of an inclusion from the GPTs.

We identify $x=(x_1,x_2) \in \RR^2$ with $z=x_1+ix_2 \in \CC$. We denote by $\Re\{\cdot\}$ and $\Im\{\cdot\}$ the real and imaginary parts of a complex number, respectively.
\begin{definition}[\cite{Ammari:2013:MSM}]
Set $P_k(z)=z^k$ for each natural number $k$.
For each $m,n=1,2,\dots$, we define the complex contracted generalized polarization tensors, which we also call the GPTs, as
\beq\label{def:NN}
\begin{aligned}
\NN_{mn}^{(1)}(\Omega, \lambda)&=\int_{\p \Omega} P_n(z)\left(\lambda I-\Kcal^*_{\p \Omega}\right)^{-1}\left[\pd{P_m }{\nu} \right](z) \,d\sigma(z),\\
\NN_{mn}^{(2)}(\Omega, \lambda)&=\int_{\p \Omega} P_n(z) \left(\lambda I-\Kcal^*_{\p \Omega}\right)^{-1}\left[\pd{\overline{P_m}}{\nu}\right](z) \,d\sigma(z).
\end{aligned}
\eeq
We denote the semi-infinite matrices $\NN^{(1)}=\big(\NN_{mn}^{(1)}\big)_{m,n=1}^\infty$ and $\NN^{(2)}=\big(\NN_{mn}^{(2)}\big)_{m,n=1}^\infty$.
\end{definition}

The GPTs, $\NN_{mn}^{(1)}$ and $\NN_{mn}^{(2)}$, are linear combinations of $M_{\alpha\beta}$, whose expansion coefficients are from the expansion of the complex polynomials into real polynomials. The values of the GPTs can be obtained from multistatic measurements \cite{Ammari:2007:PMT}. 

In this paper, we consider the inverse problem of recovering the inclusion $\Om$ and its conductivity $\sigma_c$ (equivalently, $\lambda = \frac{\sigma_c + \sigma_m}{2(\sigma_c - \sigma_m)}$) from $\NN_{mn}^{(1)}$ and $\NN_{mn}^{(2)}$.

\subsection{Faber polynomial polarization tensors (FPTs)}\label{FG}

We remind the reader that $\Om$ is a planar simply connected bounded domain. We now consider $\Om$ as a domain in the complex plane $\CC$. From the Riemann mapping theorem, there uniquely exist $\gamma > 0$ and a conformal mapping $\Psi$ from $\{w\in\CC:|w|>\gamma\}$ onto $\CC\setminus \overline{\Omega}$ such that
\beq\label{conformal:Psi}
\Psi(w)=w+a_0+\frac{a_1}{w}+\frac{a_2}{w^2}+\cdots.
\eeq 
One can numerically compute $\gamma$ and $a_n$ for a given domain $\Om$ by solving a boundary integral equation \cite{Jung:2021:SEL,Wala:2018:CMD}.

As a univalent function, $\Psi$ defines the so-called Faber polynomials $\{F_m\}_{m=1}^\infty$ \cite{Faber:1903:UPE}, which form a basis for complex analytic functions in $\Omega$, by the relation 
\beq\label{eqn:Fabergenerating}
\frac{w\Psi'(w)}{\Psi(w)-z}=\sum_{m=0}^\infty \frac{F_m(z)}{w^{m}},\quad z\in{\overline{\Omega}},\ |w|>\gamma.
\eeq
The Faber polynomials $F_m$ are monic polynomials of degree $m$ that are uniquely determined by the conformal mapping coefficients $\{ a_n \}_{0 \le n \le m-1}$ via the recursive relation (see, for instance, \cite{Duren:1983:UF})
\beq\label{Faberrecursion}
F_{m+1} (z) = z F_m (z) - m a_m - \sum_{n=0} ^{m} a_n F_{m-n} (z), \quad m\ge 0.
\eeq 
In particular, we have
\be
\begin{aligned}
F_0(z)&=1,\quad
F_1(z)=z-a_0,\quad
F_2(z)=z^2-2a_0 z+a_0^2-2a_1,\\
F_3(z)& = z^3 - 3a_0z^2 + 3(a_0^2-a_1)z - a_0^3+3a_0a_1-3a_2.
%F_4(z)& = z^4 - 4a_0z^3 + (6a_0^2 - 4a_1)z^2 - 4(a_0^3 - 2a_0a_1 + a_2)z + a_0^4 - 4a_0^2 a_1 + 4a_0a_2 + 2a_1^2 - 4a_3.
\end{aligned}
\ee

\begin{definition}[\cite{Choi:2018:GME}]\label{FPT}
For each $m,n=1,2,\dots$, we define the Faber polynomial polarization tensors (FPTs) as 
\begin{align*}
\FF_{mn}^{(1)}(\Om, \lambda)&=\int_{\p\Om}F_n(z)\left(\lambda I-\Kcal^*_{\p\Om}\right)^{-1}\left[\pd{F_m}{\nu}\right](z)\,d\sigma(z),\\
\FF_{mn}^{(2)}(\Om, \lambda)&=\int_{\p\Om}F_n(z)\left(\lambda I-\Kcal^*_{\p\Om}\right)^{-1}\left[\pd{\overline{F_m}}{\nu}\right](z)\,d\sigma(z).
\end{align*}
We denote the semi-infinite matrices $\FF^{(1)}=\big(\FF_{mn}^{(1)}\big)_{m,n=1}^\infty$ and $\FF^{(2)}=\big(\FF_{mn}^{(2)}\big)_{m,n=1}^\infty$.
\end{definition}

We can express the Faber polynomial $F_m(z)$ as
\beq\label{Fm:pmn}
F_m(z) = \sum_{n=0}^{m} p_{mn} z^n,
\eeq
where for a fixed $m$, the coefficient $p_{mn}$ depends only on $\{ a_k \}_{0 \le k \le m-1}$. One can easily obtain recursive formulas for $p_{mn}$ from \eqnref{Faberrecursion}. 
From \eqnref{Fm:pmn} and the definition of the FPTs, it holds for each $m,n$ that 
\beq\label{FPN}
\begin{aligned}
\FF_{mn}^{(1)} &= \sum_{k=1} ^{m} \sum_{l=1} ^{n} p_{mk} \, p_{nl}\, \mathbb{N}_{kl}^{(1)}, \\
\FF_{mn}^{(2)} &= \sum_{k=1} ^{m} \sum_{l=1} ^{n} \overline{p_{mk}}\,  p_{nl}\, \mathbb{N}_{kl}^{(2)}.
\end{aligned}
\eeq

\subsection{Grunsky coefficients}

An essential property of $F_m(z)$ is that $F_m(\Psi(w))$ has only one positive order term $w^m$. In other words,
\begin{equation*}
F_m(\Psi(w)) = w^m+\sum_{n=1}^{\infty}c_{mn}{w^{-n}},\quad |w|>\gamma,
\end{equation*}
where $c_{mn}$ are the so-called Grunsky coefficients. It holds that (see \cite{Duren:1983:UF})
\begin{equation}\label{eqn:Grunsky:sym1}
nc_{mn} =mc_{nm} \quad \mbox{for all } m,n\in\NN
\end{equation}
and
\begin{equation}\label{GC:recur}
\begin{aligned}
&c_{1m} = a_m,\quad c_{m1} = ma_m,\\
&c_{m(n+1)} = c_{(m+1)n} - a_{m+n} + \sum_{s=1}^{m-1} a_{m-s}c_{sn} - \sum_{s=1}^{n-1} a_{n-s}c_{ms},\quad m,n\geq1.
\end{aligned}
\end{equation}
We can symmetrize the Grunsky coefficients as 
\beq\label{g}
g_{mn} = \sqrt{\frac{n}{m}} \frac{c_{mn}}{\gamma^{m+n}}.
\eeq
From \eqnref{eqn:Grunsky:sym1}, it holds that
\beq\label{eqn:g:sym}
g_{mn}=g_{nm}\quad\mbox{for all } m,n \in\NN.
\eeq
We refer the reader to \cite{Duren:1983:UF} for more details on the Faber polynomials  and to \cite{Chui:1992:FSA, Curtiss:1964:HIF,Curtiss:1966:SDP, Ellacott:1983:CFS,Gao:2004:FSM, Luo:2009:FSM} for their applications in diverse areas.

We denote by $C$ (resp., $G$) the semi-infinite matrix given by the Grunsky coefficients (resp., the symmetrized  Grunsky coefficients), that is,
\begin{equation}\label{mat:G:ep}
C%=\big(c_{mn}\big)_{m,n=1}^\infty
=\left(
\begin{matrix}
\ c_{11} & c_{12} & c_{13}& \cdots\\[2mm]
\ c_{21} & c_{22} & c_{23} & \cdots\\[2mm]
\ c_{31} & c_{32} & c_{33} & \cdots\\
\ \vdots & \vdots & \vdots & \ddots
\end{matrix}\right),
\qquad
%\end{equation*}
%We also set
%\begin{equation}\label{mat:G:ep}
G%=\big(g_{mn}\big)_{m,n=1}^\infty
=\left(
\begin{matrix}
\ g_{11} & g_{12} & g_{13}& \cdots\\[2mm]
\ g_{21} & g_{22} & g_{23} & \cdots\\[2mm]
\ g_{31} & g_{32} & g_{33} & \cdots\\
\ \vdots & \vdots & \vdots & \ddots
\end{matrix}\right).
\end{equation}
From \eqnref{g}, it holds that
\beq\label{GGG}
G= \NNc^{-\frac{1}{2}}\gamma^{-\NNc} C\gamma^{-\NNc}\NNc^{\frac{1}{2}},
\eeq
where we set
\begin{equation}\label{eqn:matrix}
\gamma^{\pm\NNc}=
\left(
\begin{matrix}
\  \gamma^{\pm1} & 0 & 0 &\cdots  \\[2mm]
\ 0 & \gamma^{\pm2} & 0 & \cdots\\[2mm]
\ 0 & 0 & \gamma^{\pm3} & \cdots \\
\ \vdots & \vdots &\vdots  & \ddots 
\end{matrix}
\right),
\qquad
\NNc^{\pm\frac{1}{2}}=
\left(
\begin{matrix}
\ 1  & 0 & 0 &\cdots \\[2mm]
\ 0 & {\sqrt{2}}^{\, \pm 1}  & 0 & \cdots\\[2mm]
\ 0 & 0 & {\sqrt{3}}^{\, \pm 1} & \cdots\\
\ \vdots & \vdots &\vdots  & \ddots 
\end{matrix}
\right).
\end{equation}
Similarly to equation \eqnref{eqn:matrix}, the matrix $\gamma^{\pm 2\NNc}$ (resp., $\NNc$ and $\NNc^{-1}$) denotes the diagonal matrix whose $(n,n)$-entries are $\gamma^{\pm 2n}$ (resp., $n$ and $n^{-1}$).

\section{Previous studies}

We review previous studies on the shape recovery of a planar conductivity inclusion by using the concept of the GPTs. The first direction is to derive explicit expressions for the conformal mapping coefficients of the inclusion in terms of the GPTs, assuming that the inclusion has extreme conductivity, that is, the inclusion is either insulating or perfectly conducting (see Subsection \ref{subsec:recover:extreme}). The second direction is to adopt an optimization approach for an inclusion with arbitrary constant conductivity (see Subsection \ref{optim}).

\subsection{Conformal mapping recovery for the extreme conductivity case}\label{subsec:recover:extreme}

The exterior conformal mapping $\Psi$ associated with $\Om$ extends to the boundary of $\Om$ as a homeomorphism by the Caratheodory extension theorem \cite{Caratheodory:1913:GBR}.
In particular, $\Psi$ gives a natural parameterization for $\partial \Omega$ and, in particular, determines the shape of $\Om$. 

For an inclusion with extreme conductivity, the multipole expansion of $u$ admits an extension up to $\p\Om$ on which the Dirichlet or Neumann boundary condition is prescribed. For the case $\sigma_c=\infty$, $u$ determines a holomorphic function $U(z)$ satisfying $u(x)=\Re\{U(z)\}$ and, thus,
\beq\notag%\label{BC:infty}
\Re\{U(z)\}=\mbox{constant on }\p\Om.
\eeq
For the case $\sigma_c=0$, it holds that
\beq\label{BC:zero}
\Im\{U(z)\}=\mbox{constant on }\p\Om.
\eeq
Using this relation, the coefficients of $\Psi$ were explicitly expressed by the GPTs for an inclusion with $\sigma_c=0$ \cite{Choi:2018:CEP,Kang:2015:CCM}.
Similar results could be derived for the perfectly conducting case by considering a harmonic conjugate of $u$.

The layer potential operators associated with $\Om$ admit infinite series expansions with respect to basis functions defined with $\Psi$ \cite{Jung:2021:SEL}. For an inclusion with a $C^{1,\alpha}$ boundary, one can then solve the conductivity inclusion problem by using these series expansions. As an application, we can express the FPTs with the Grunsky coefficients as follows.
\begin{lemma}[\cite{Choi:2018:GME}]\label{thm:FPT}
Let $\Om$ have a $C^{1,\alpha}$ boundary.
For each $m,n$, it holds that
\begin{align*}
\FF_{mn}^{(1)} (\Om, \lambda)
&=  4\pi n c_{mn} + 4\pi n \left( 1 - 4\lambda^2 \right) \left(C\left( 4\lambda^2 I - \gamma^{-2\NNc}\, \overline{C}\gamma^{-2\NNc} C \right)^{-1} \right)_{mn},\\
\FF_{mn}^{(2)}(\Om, \lambda)
&= 8\pi n \lambda \gamma^{2m}\delta_{mn} + 8 \pi n \lambda \gamma^{2m} \left( 1 - 4\lambda^2 \right) \left( \left( 4\lambda^2 I - \gamma^{-2\NNc}\, \overline{C}\gamma^{-2\NNc} C \right)^{-1} \right)_{mn}.
\end{align*}
Here, $\delta_{mn}$ is the Kronecker delta function.
\end{lemma}
For the case $\lambda=\pm\frac{1}{2}$, Lemma \ref{thm:FPT} implies that
\begin{align}
&\FF_{m1}^{(1)}\left(\Om, \lambda \right) = 4\pi c_{m1} \quad \mbox{for }m\geq 1,\label{FPT:direct1}\\
&\FF_{11}^{(2)}\left(\Om,\lambda \right) = 8\pi \lambda \gamma^2, \quad \mathbb{F}_{21}^{(2)} \left(\Om, \lambda \right) = 0.\label{FPT:direct2}
\end{align}
Using these relations and \eqnref{GC:recur} for $n=1$, one can completely recover the conformal mapping coefficients, where the expression formulas are much simpler than those derived in \cite{Choi:2018:CEP,Kang:2015:CCM} as follows.
\begin{theorem}[\cite{Choi:2021:ASR}]\label{thm:exact0}
Let $\Om$ be a simply connected, bounded $C^{1,\alpha}$ domain with $\sigma_c = 0$ or $\infty$ (equivalently, $\lambda = -\frac{1}{2}$ or $\frac{1}{2}$).
The exterior conformal mapping $\Psi$ associated with $\Om$ (see \eqnref{conformal:Psi}) satisfies
\begin{align*}
\gamma^2 &= \frac{\lambda}{2\pi}\, \mathbb{N}_{11}^{(2)}(\Om,\lambda), \quad a_0 = \frac{\mathbb{N}_{12}^{(2)} (\Om,\lambda ) }{2\,\mathbb{N}_{11} ^{(2)} (\Om,\lambda) }, \\
 a_m &=\frac{\lambda^2}{\pi m} \sum_{n=1} ^{m} p_{mn}\, {\mathbb{N}_{n1} ^{(1)} (\Om,\lambda)},\quad m\ge 1,
\end{align*}
where $p_{m1}, p_{m2}, \cdots, p_{mm}$ denote the coefficients of $F_m(z)$ of $\Om$ as defined in \eqnref{Fm:pmn}. 
In particular, each $a_m$ is uniquely determined by $\mathbb{N}_{12}^{(2)}$ and $\{\mathbb{N}_{n1}^{(1)}\}_{1\le n \le m}$. 
\end{theorem}

\subsection{Optimization approach for the arbitrary conductivity case}\label{optim}

Let $D$ be an inclusion having the conductivity $\sigma_c$ with a $C^2$ boundary given by a small perturbation of $D_0$, that is,
\beq\label{Om:deform}
\p D = \left\{ x + \ep f(x) \nu_0(x) : \, x\in \p D_0 \right\}
\eeq
with a real-valued function $f\in C^1(\p D_0)$ and a small parameter $\ep>0$, 
where $\nu_0$ is the outward unit normal vector to $\p D_0$. 
It then holds that (see \cite{Ammari:2012:GPT}) 
\beq\label{eqn:GPT:asymp}
\begin{aligned}
	&\sum_{\alpha, \beta} a_\alpha b_\beta M_{\alpha \beta} (D,\lambda) - \sum_{\alpha, \beta} a_\alpha b_\beta M_{\alpha \beta} (D_0,\lambda) 	\\
=&\, \ep\, \Big(\frac{\sigma_c}{\sigma_m}-1\Big) \int_{\p \Omega} f(x) \bigg[ \frac{\p v}{\p \nu} \Big|^- \frac{\p u}{\p \nu} \Big|^- + \frac{\sigma_m}{\sigma_c} \frac{\p v}{\p T} \Big|^- \frac{\p u}{\p T} \Big|^- \bigg] (x)\, d \sigma(x) + O(\ep^2),
	\end{aligned}
	\eeq
	where $u$ and $v$ are the solutions to 
	\beq\label{cond_eqn1}
	\begin{cases}
		\ds\Delta u=0\quad&\mbox{in } D_0\cup (\RR^2\setminus\overline{D_0}), \\
		\ds u\big|^+=u\big|^-\quad&\mbox{on }\p D_0,\\[1mm]
		\ds \sigma_m\pd{u}{\nu}\Big|^+=\sigma_c\pd{u}{\nu}\Big|^-\quad&\mbox{on }\p D_0,\\[1mm]
		\ds u(x) - H_1(x)=O({|x|^{-1}})\quad&\mbox{as } |x| \to \infty
	\end{cases}
\end{equation}
and
\beq\label{cond_eqn2}
\begin{cases}
	\ds\Delta v=0\quad&\mbox{in } D_0\cup (\RR^2\setminus\overline{D_0}), \\
	\ds \sigma_c v\big|^+=\sigma_m v\big|^-\quad&\mbox{on }\p D_0,\\[1mm]
	\ds \pd{v}{\nu}\Big|^+=\pd{v}{\nu}\Big|^-\quad&\mbox{on }\p D_0,\\[1mm]
	\ds v(x) - H_2(x)  =O({|x|^{-1}})\quad&\mbox{as } |x| \to \infty
\end{cases}
\end{equation}
with entire harmonic functions  $H_1(x)=\sum_\alpha a_\alpha x^\alpha$ and $H_2(x)=\sum_\beta b_\beta x^\beta$.

Iterative methods have been developed for approximating the shape of an inclusion $\Om$ by adopting an optimization approach, where the cost function for a test domain $D$ has the form (with a fixed positive integer $K$)  \cite{Ammari:2014:GPT,Ammari:2012:GPT}
$$
J[D] = \frac{1}{2} \sum_{|\alpha|+|\beta| \le K} \Bigg|\sum_{\alpha, \beta} a_{\alpha} b_{\beta} M_{\alpha \beta}  (D,\lambda) - \sum_{\alpha, \beta} a_{\alpha} b_{\beta} M_{\alpha \beta} (\Omega,\lambda) \Bigg|^2.
$$
Equation \eqnref{eqn:GPT:asymp} provides the shape derivative for the cost function.

If $D_0$ is a disk, one can simply solve \eqnref{cond_eqn1} and \eqnref{cond_eqn2} and, by rewriting \eqnref{eqn:GPT:asymp}, derive asymptotic formulas for the Fourier coefficients of the shape perturbation function $f$ as elementary functions of the GPTs (see \cite{Ammari:2010:CIP}). By using the asymptotic formulas for $f$, one can non-iteratively approximate an inclusion $\Om$ by considering it as a small perturbation of an equivalent disk, where we set $D_0$ as the equivalent ellipse and $f$ as the perturbation from $\p D_0$ to $\p\Om$. 

If $D_0$ is an ellipse, the integral formula in \eqnref{eqn:GPT:asymp} is too complicated in Cartesian coordinates to find an explicit analytic form. In \cite{Choi:2021:ASR}, the curvilinear orthogonal coordinates and the Faber polynomials associated with the ellipse were successfully employed to derive explicit asymptotic formulas for the integral in \eqnref{eqn:GPT:asymp}. These asymptotic formulas (by taking $D_0$ as an equivalent ellipse) allow us to non-iteratively approximate an inclusion with arbitrary conductivity of general shape, including a straight or asymmetric shape (see \cite{Choi:2021:ASR} for the details).

\section{Reconstruction of a smooth inclusion}

For an inclusion with arbitrary constant conductivity, the boundary value of $u$ in \eqnref{transmission} is no longer explicit. Thus, it is a challenge to generalize Theorem \ref{thm:exact0} to the arbitrary constant conductivity case. In this section, we derive factorization formulas for two semi-infinite matrices whose entries are scalar-valued complex contracted GPTs and use the formulas to provide an answer to this problem as the primary conclusion of this paper. 

\subsection{Matrix factorizations for the GPTs}

We denote by $I$ the semi-infinite identity matrix. From \eqnref{GGG}, it holds that
$$ 4\lambda^2 I - \gamma^{-2\NNc}\, \overline{C}\gamma^{-2\NNc} C=\gamma^{-\NNc}\NNc^{\frac{1}{2}} \left( 4\lambda^2 I - \overline{G} G \right)\gamma^\NNc \NNc^{-\frac{1}{2}}.$$
Lemma \ref{thm:FPT} and \eqnref{GGG} lead to matrix factorizations for the FPTs.
\begin{lemma}\label{lemma:FPT:factor}
Let $\Om$ have a $C^{1,\alpha}$ boundary. 
The FPTs of $\Om$ admit the matrix factorizations
\beq\label{FPT:factor}
\begin{aligned}
\FF^{(1)}& = 4\pi \,\gamma^\NNc \NNc^{\frac{1}{2}}\, G \left[ I+ (1-4\lambda^2) \left( 4\lambda^2 I - \overline{G} G \right)^{-1} \right] \NNc^{\frac{1}{2}}\, \gamma^\NNc,  \\
\FF^{(2)}&= 8\pi \lambda\, \gamma^{\NNc} \NNc^{\frac{1}{2}} \left[ I + (1-4\lambda^2) \left( 4\lambda^2 I - \overline{G} G \right)^{-1}  \right] \NNc^{\frac{1}{2}}\, \gamma^{\NNc}.
\end{aligned}
\eeq
\end{lemma}

From the relations of the GPTs and FPTs, we can then obtain matrix factorizations for the GPTs.
Set $P=(p_{mn})_{m,n=1}^\infty$ with $p_{mn}$ given by \eqnref{Fm:pmn}. 
We can rewrite relations \eqnref{FPN} in matrix form as 
\beq\label{FPNP}
\begin{aligned}
\FF^{(1)} &= P\, \mathbb{N}^{(1)} P^T,\\
\FF^{(2)} &= \overline{P}\, \mathbb{N}^{(2)} P^T,
\end{aligned}
\eeq
where $\overline{P}$ and $P^T$ denote the conjugate and transpose matrices of $P$, respectively.
From \eqnref{Faberrecursion}, one can easily find that, for each $m\geq1$,
\begin{align}\label{inititalP}
&p_{mm} = 1,\quad p_{(m+1)m} = - (m+1) a_0, \quad
p_{mn}=0\quad\mbox{for all }n\geq m+1.
\end{align}
Indeed,
\beq\label{def:P}
P=
\left(
\begin{matrix}
1 & 0 & 0   \quad  \cdots\\[1.5mm]
-2a_0 & 1 & 0   \quad \cdots\\[1.5mm]
3a_0^2-3a_1 & -3a_ 0 &  1  \quad \cdots\\[1.5mm]
\vdots & \vdots &  \vdots \quad \ddots 
\end{matrix}
\right).
\eeq
Hence, $P$ is lower triangular and invertible. Similarly, $\overline{P}$ and $P^T$ are invertible. From Lemma \ref{lemma:FPT:factor} and \eqnref{FPNP}, we have the following theorem.
\begin{theorem}\label{Factorization:smooth}
Let $\Om$ have a $C^{1,\alpha}$ boundary. 
The GPTs of $\Om$ admit the matrix factorizations
\begin{align}
 \mathbb{N}^{(1)} & = 4\pi \,P^{-1} \gamma^\NNc \NNc^{\frac{1}{2}}\, G \left[ I+ (1-4\lambda^2) \left( 4\lambda^2 I - \overline{G} G \right)^{-1} \right] \NNc^{\frac{1}{2}}\, \gamma^\NNc \,(P^T)^{-1}, \label{pnonep}\\
 \mathbb{N}^{(2)} &= 8\pi \lambda\, \overline{P}^{-1} \gamma^{\NNc} \NNc^{\frac{1}{2}}\, \left[ I + (1-4\lambda^2) \left( 4\lambda^2 I - \overline{G} G \right)^{-1}  \right] \NNc^{\frac{1}{2}}\, \gamma^{\NNc} \,(P^T)^{-1}. \label{pntwop}
\end{align}
\end{theorem}

We note that the GPTs are expressed in terms of the exterior conformal mapping and the conductivity value of the inclusion. The matrix factorizations \eqnref{pnonep} and \eqnref{pntwop} have one common factor depending on $\lambda$, which satisfies
\beq\label{commonfactor}
I+ (1-4\lambda^2) \left( 4\lambda^2 I -  \overline{G} G \right)^{-1}= \left( I - \overline{G} G \right) \left( 4\lambda^2 I - \overline{G} G \right)^{-1}.
\eeq

In the instance that $\lambda=\pm\frac{1}{2}$ (that is, the insulating or perfectly conducting case), the common factor \eqnref{commonfactor} is the identity matrix. It then follows the explicit expressions of the conformal mapping coefficients of $\Om$ in Theorem \ref{thm:exact0}. If $\lambda\neq\pm\frac{1}{2}$, it becomes more complicated to derive explicit formulas for the shape of the inclusion from \eqnref{pnonep} and \eqnref{pntwop}. 

\subsection{Inversion formula}
Our main idea to eliminate the common factor \eqnref{commonfactor} between $\mathbb{N}^{(1)}$ and $\mathbb{N}^{(2)}$ is to consider 
$$\mathbb{N}^{(1/2)}:= \mathbb{N}^{(1)} \big(\mathbb{N}^{(2)} \big)^{-1}.$$
We modify $\NN^{(1)}$ and $\NN^{(2)}$ as 
\beq\label{def:tildeNN}
\widetilde{\NN}^{(1)}= \mathbb{N}^{(1/2)} \,\MM\, \NN^{(2)},\quad
 \widetilde{\NN}^{(2)}:=\MM\, \NN^{(2)}
\eeq
with
\beq\notag
\MM=\left(I-\overline{\mathbb{N}^{(1/2)} }\,\mathbb{N}^{(1/2)}\right)\left(I-4\lambda^2\,\overline{\mathbb{N}^{(1/2)} }\,\mathbb{N}^{(1/2)}\right)^{-1}.
\eeq
Let $\widetilde{\NN}^{(1)}_{mn}$ and $\widetilde{\NN}^{(2)}_{mn}$ denote the $(m,n)$-component of $\widetilde{\NN}^{(1)}$ and $\widetilde{\NN}^{(2)}$, respectively. 
If $\lambda=\pm\frac{1}{2}$, then $\MM=I$ and $\widetilde{\NN}^{(j)}=\NN^{(j)}$, $j=1,2$. 
For $\widetilde{\NN}^{(2)}$, the same expression of ${\NN}^{(2)}$ in \eqnref{pntwop} with extreme conductivity holds except the constant multiplication as follows. 
\begin{lemma}\label{lemma:tildeN2}
For arbitrary constant $\sigma_c$ satisfying $0<\sigma_c\neq\sigma_m<\infty$, it holds that
$$\widetilde{\NN}^{(2)}
=\frac{2\pi}{\lambda}\,\overline{P}^{-1}\,\gamma^{2\NNc} \NNc \,(P^T)^{-1}.$$
\end{lemma}
\begin{proof}

By combining \eqnref{pnonep} and \eqnref{pntwop}, we obtain
\beq\label{Nhalf}
\mathbb{N}^{(1/2)} = (2\lambda)^{-1} P^{-1} \gamma^\NNc \NNc^{\frac{1}{2}}\, G \,\NNc^{-\frac{1}{2}} \gamma^{-\NNc}\, \overline{P} = (2\lambda)^{-1} P^{-1} C\, \gamma^{-2\NNc} \, \overline{P}
\eeq
and, thus, 
\beq\label{Cformula}
G=2\lambda \,\NNc^{-\frac{1}{2}} \gamma^{-\NNc}\,P\,\NN^{(1/2)}\,\overline{P}^{-1}\,\gamma^{\NNc} \NNc^{\frac{1}{2}}.
% \gamma^{\NN} \wtC  \gamma^{\NN} = 2 \lambda P\, \NN^{(1/2)}\, \overline{P}^{-1} \gamma^{2\NN}.
\eeq
It then follows that
\beq\label{Cformula2}
\overline{G} G = 4\lambda^2\, \NNc^{-\frac{1}{2}} \gamma^{-\NNc} \,\overline{P}\, \overline{\NN^{(1/2)}}\,\NN^{(1/2)}\,\overline{P}^{-1}\gamma^{\NNc} \NNc^{\frac{1}{2}}.
\eeq
Substituting \eqnref{Cformula2} into \eqnref{pntwop} (also using \eqnref{commonfactor}), we derive
\begin{align}\notag
\mathbb{N}^{(2)} 
= &8\pi \lambda \,\overline{P}^{-1} \gamma^{\NNc} \NNc^{\frac{1}{2}} \left( I - \overline{G} G \right) \left( 4\lambda^2 I - \overline{G} G \right)^{-1}\NNc^{\frac{1}{2}}\, \gamma^{\NNc} \,(P^T)^{-1}\\ \notag
= &\frac{2\pi}{\lambda} \left( I - 4\lambda^2\, \overline{\NN^{(1/2)}}\,\NN^{(1/2)} \right) \left(  I - \overline{\NN^{(1/2)}}\,\NN^{(1/2)} \right)^{-1}\,\overline{P}^{-1}\gamma^{2\NNc} (\NNc^{\frac{1}{2}} )^2 \,(P^T)^{-1}.
\end{align}
This completes the proof.
\end{proof}

We now generalize the formula for the extreme conductivity case in Theorem \ref{thm:exact0} to the arbitrary conductivity case in terms of $\widetilde{\NN}^{(2)}$. If the GPTs of the inclusion are fully given, we can recover the exterior conformal mapping of the inclusion.

\begin{theorem}\label{conformalrecovery}
Let $\Om$ be a simply connected, planar $C^{1,\alpha}$ domain with arbitrary constant conductivity $\sigma_c$ satisfying $0<\sigma_c\neq\sigma_m<\infty$  (that is, $\lambda = \frac{\sigma_c + \sigma_m}{2(\sigma_c - \sigma_m)}$ is an arbitrary real number satisfying $|\lambda|>\frac{1}{2}$). 
Let $\widetilde{\NN}^{(j)}=\widetilde{\NN}^{(j)}(\Om,\lambda)$, $j=1,2$, be given by \eqnref{def:tildeNN}. Then,
\begin{itemize}
\item[\rm(a)]

$\lambda$ satisfies the implicit equation
\beq\label{lambda:GPTs}
\lambda  = \pi\,\frac{ \widetilde{\NN}^{(2)}_{11} \, \widetilde{\NN}^{(2)}_{22} -  \widetilde{\NN}^{(2)}_{12} \, \widetilde{\NN}^{(2)}_{21} }{\left( \widetilde{\NN}^{(2)}_{11}\right)^3}, \mbox{ and}
\eeq

\item[\rm(b)]
the exterior conformal mapping coefficients associated with $\Om$ satisfy the explicit formulas
\beq\label{conformal:GPTs}
\begin{aligned}
\gamma^2 &= \frac{\lambda}{2\pi} \widetilde{\NN}^{(2)}_{11}(\Om,\lambda), \quad a_0 = \frac{\widetilde{\NN}^{(2)}_{12}(\Om,\lambda)}{2 \widetilde{\NN}^{(2)}_{11}(\Om,\lambda)}, \\\quad a_m &= \frac{\lambda^2}{\pi m} \sum_{n=1}^m p_{mn} \widetilde{\NN}^{(1)}_{n1}(\Om,\lambda), \quad m\ge 1,
\end{aligned}
\eeq
where $p_{m1}, p_{m2}, \cdots, p_{mm}$ denote the coefficients of $F_m(z)$ of $\Om$ as defined in \eqnref{Fm:pmn}. 
In particular, each $a_m$ is uniquely determined by $\lambda$, $\widetilde{\NN}_{12}^{(2)}$ and $\{\widetilde{\NN}_{n1}^{(1)}\}_{1\le n\le m}$.
\end{itemize} 
\end{theorem}

\begin{proof}
From Lemma \ref{lemma:tildeN2}, we have
\beq\label{gammatwoNN}
\gamma^{2\NNc}
= \frac{\lambda}{2\pi} \overline{P}\, \widetilde{\NN}^{(2)} P^T \NNc^{-1}.
\eeq
From \eqnref{inititalP} and \eqnref{gammatwoNN}, we compute the $(1,1)$- and $(1,2)$-elements of $\gamma^{2\NNc}$:
\begin{align}
\gamma^2 &= \big[ \gamma^{2\NNc} \big]_{11} = \frac{\lambda}{2\pi}\, \widetilde{\NN}^{(2)}_{11},\label{gammasquare}\\[2mm]
0 &= \big[ \gamma^{2\NNc} \big]_{12}= \frac{\lambda}{4\pi}\big[\, \overline{P} \,\widetilde{\NN}^{(2)} P^T\, \big]_{12}= \frac{\lambda}{4\pi} \,\widetilde{\NN}^{(2)}_{12} -\frac{\lambda }{2\pi}a_0\, \widetilde{\NN}^{(2)}_{11},\notag
\end{align}
and this implies that
\beq\label{center}
a_0 = \frac{ \widetilde{\NN}^{(2)}_{12}}{2 \widetilde{\NN}^{(2)}_{11}}.
\eeq
Let us continue to find $a_m$. We now substitute \eqnref{gammatwoNN} into \eqnref{Nhalf} and get
\beq\notag%\label{Cformula2}
C = \frac{\lambda^2}{\pi} P\, \NN^{(1/2)} \, \widetilde{\NN}^{(2)}\, P^T \NNc^{-1}.
\eeq
From \eqnref{GC:recur}, we obtain
\be 
a_m =  \frac{c_{m1}}{m} = \frac{\lambda^2}{\pi m} \sum_{n=1}^m p_{mn} \big[ \NN^{(1/2)} \widetilde{\NN}^{(2)} \big]_{n1}\quad\mbox{for each }m\geq1.
\ee

For the constraint equation of $\lambda$, we compute the $(2,2)$-element of \eqnref{gammatwoNN} by using \eqnref{inititalP}, \eqnref{gammasquare}, and \eqnref{center}:
\begin{align*}
\frac{\lambda^2}{4\pi^2} \big( \widetilde{\NN}^{(2)}_{11}\big)^2 
&= \big[ \gamma^{2\NNc} \big]_{22}\\
&= \frac{\lambda}{4\pi} \big[\overline{P}\,  \widetilde{\NN}^{(2)} P^T \big]_{22}\\
&= \frac{\lambda}{4\pi} \left( 4a_0\overline{a_0}\,  \widetilde{\NN}^{(2)}_{11} -2a_0 \widetilde{\NN}^{(2)}_{21} -2\overline{a_0}\, \widetilde{\NN}^{(2)}_{12} +  \widetilde{\NN}^{(2)}_{22} \right)\\
&= \frac{\lambda}{4\pi} \left(  \widetilde{\NN}^{(2)}_{22} - \frac{ \widetilde{\NN}^{(2)}_{12}}{ \widetilde{\NN}^{(2)}_{11}}\,  \widetilde{\NN}^{(2)}_{21} \right).
\end{align*}
Since $\lambda$ is nonzero, we finally get \eqnref{lambda:GPTs}. 
\end{proof}
We note that the modified GPTs are defined by using $\lambda$ as well as the GPTs. 
The right-hand side of \eqnref{lambda:GPTs} also depends on $\lambda$. We can numerically find the value of $\lambda$ by an iterative algorithm as will be shown in Section \ref{sec:numerical}.  
Then, the conformal radius $\gamma$ and the coefficients $a_m$ follow from $\widetilde{\NN}^{(1)}$, $\widetilde{\NN}^{(2)}$ and the computed value of $\lambda$.

\begin{remark}
According to Theorem \ref{conformalrecovery}, when the inclusion has finite conductivity, we need to know all the values of the GPTs to find $a_m$ for a fixed $m$. However, if the conductivity is extreme, i.e., $\sigma_c=0$ or $\infty$, we only need the GPTs of finitely many indices as in Theorem \ref{thm:exact0}. 
\end{remark}

%%%%%%%%%%%%%%%%%%%%%%%
\section{Extension to a Lipschitz inclusion}

We now generalize Theorems \ref{Factorization:smooth} and \ref{conformalrecovery} to Lipschitz domains that satisfy the following shrinkable property with $s=a_0$. 
\begin{definition}[Star-shaped domain]
A set $D\subset\RR^2$ is called a star-shaped domain with respect to a point $s_0\in D$ if $\mu(D -s_0) + s_0 \Subset D$ for all $\mu \in [0,1)$.
\end{definition}

For a closed Jordan curve $\Gamma$ in $\CC$, we say that $\Gamma $ is smooth if it admits a parameterization $z(t):[0,2\pi)\rightarrow\CC$ such that $z'(t)$ is continuous and $\neq 0$, following the definition in \cite[Chapter 3.2]{Pommerenke:1992:BBC}. 
A piecewise smooth Jordan curve without cusps is quasiconformal (refer to \cite{Ahlfors:1963:QR} and \cite[Chapter 5.4] {Pommerenke:1992:BBC} for the characterization of a quasiconformal curve). 
According to \cite[Theorem 9.14]{Pommerenke:1975:UF}, 
it holds that $\left\lVert G \right\rVert_{l^2\rightarrow l^2} \le \kappa$ for some $\kappa\in[0,1)$ if and only if $\p \Om$ is quasiconformal.
In particular, the matrix $4\lambda^2 I - \overline{G} G$ is invertible for all $|\lambda|\geq \frac{1}{2}$.

The following theorems are the main results in this section. We provide the proof of Theorem \ref{conformal2GPT} at the end of Subsection \ref{proof:thm:F1F2Lipschitz}. 
\begin{theorem}[Factorizations of the GPTs for a Lipschitz inclusion] \label{conformal2GPT} 
Let $\Om\subset\RR^2$ be a simply connected, bounded, and Lipschitz domain with arbitrary constant conductivity $\sigma_c$ satisfying $0<\sigma_c\neq\sigma_m<\infty$,
where $\p\Om$ is a piecewise smooth Jordan curve without cusps.
Assume that $\Om$ is a star-shaped domain with respect to $a_0$,  where $a_0$ is the constant coefficient of the conformal mapping $\Psi$ corresponding to $\Om$.
Then the GPTs of $\Om$ admit the matrix factorizations \eqnref{pnonep} and \eqnref{pntwop}. 
\end{theorem}

In view of the derivation of \eqnref{lambda:GPTs} and \eqnref{conformal:GPTs}, Theorem \ref{conformal2GPT} directly leads to the following result for a Lipschitz inclusion. 
\begin{theorem}\label{thm:recovery:Lip}
Under the same assumptions for $\Om$ as in Theorem \ref{conformal2GPT}, the shape recovery formulas \eqnref{lambda:GPTs} and \eqnref{conformal:GPTs} in Theorem \ref{conformalrecovery} hold. 
\end{theorem}

\subsection{Shape monotonicity of the GPTs}\label{sec:mono:GPT}
Harmonic combinations of the GPTs admit shape monotonicity:
\begin{lemma}[\cite{Ammari:2005:PTT}]\label{monotonicity}
Let $\Omega' \subsetneq \Omega$ and $J$ be a finite multi-index set. Let $a_\alpha$ be real-valued constant coefficients such that $h(x)=\sum_{\alpha\in J} a_\alpha x^\alpha$ is a harmonic polynomial. Then, we have
\begin{align}\label{monotone1}
\sum_{\alpha,\beta\in J} a_\alpha a_\beta M_{\alpha \beta}(\Omega,\lambda)& > \sum_{\alpha,\beta\in J} a_\alpha a_\beta M_{\alpha \beta}(\Omega',\lambda) \quad \mbox{if } \lambda > \frac{1}{2},\\
\label{monotone2}
\sum_{\alpha,\beta\in J} a_\alpha a_\beta M_{\alpha \beta}(\Omega,\lambda) &< \sum_{\alpha,\beta\in J} a_\alpha a_\beta M_{\alpha\beta}(\Omega',\lambda) \quad \mbox{if } \lambda < -\frac{1}{2}.
\end{align}
\end{lemma}

The following monotonicity property of the FPTs will be essentially used to prove Theorem \ref{conformal2GPT} in Subsection \ref{proof:thm:F1F2Lipschitz}.
\begin{lemma}\label{diagonalmonotone}
The linear combinations of the FPTs 
\beq\label{def:oper:A}
\begin{aligned}
&\mathbb{A}_{mn}^{(\pm1)}:=\text{Re}\left\{ \FF_{mm}^{(1)} + \FF_{nn}^{(1)} + \FF_{mm}^{(2)} + \FF_{nn}^{(2)}\right\} \pm 2 \text{Re}\left\{\FF_{mn}^{(1)} +\FF_{mn}^{(2)} \right\},\\
&\mathbb{A}_{mn}^{(\pm2)}:=\text{Re}\left\{- \FF_{mm}^{(1)} - \FF_{nn}^{(1)}+ \FF_{mm}^{(2)} + \FF_{nn}^{(2)} \right\} \mp 2 \text{Re}\left\{\FF_{mn}^{(1)} - \FF_{mn}^{(2)} \right\},\\
&\mathbb{A}_{mn}^{(\pm3)}:=\text{Re}\left\{ \FF_{mm}^{(1)} - \FF_{nn}^{(1)} + \FF_{mm}^{(2)} + \FF_{nn}^{(2)} \right\} \mp 2 \text{Im}\left\{\FF_{mn}^{(1)} + \FF_{mn}^{(2)} \right\},\\
&\mathbb{A}_{mn}^{(\pm4)}:=\text{Re}\left\{ -\FF_{mm}^{(1)} + \FF_{nn}^{(1)} + \FF_{mm}^{(2)} + \FF_{nn}^{(2)}\right \} \pm 2 \text{Im}\left\{\FF_{mn}^{(1)} - \FF_{mn}^{(2)}\right\}
\end{aligned}
\eeq
have increasing and decreasing monotonicity with respect to the domain if $\lambda>\frac{1}{2}$ and $\lambda<-\frac{1}{2}$, respectively. 
\end{lemma}
\begin{proof}
From the definition of the FPTs and the symmetry of the GPTs, we obtain
\begin{align}
&\text{Re}\left\{ \FF_{mm}^{(1)} + \FF_{nn}^{(1)} + \FF_{mm}^{(2)} + \FF_{nn}^{(2)}\right\} \pm 2 \text{Re}\left\{\FF_{mn}^{(1)} +\FF_{mn}^{(2)} \right\} \notag\\
&=  2\int_{\p\Om}\text{Re}\left\{F_m \pm F_n\right\}\left(\lambda I-\Kcal^*_{\p\Om}\right)^{-1}\left[\pd{}{\nu}\text{Re}\left\{F_m \pm F_n\right\}\right] d\sigma.\notag
\end{align}
From Lemma \ref{monotonicity} with $h = \text{Re}\{F_m \pm F_n\}$, we have the monotonicity for $\text{Re}\{ \FF_{mm}^{(1)} + \FF_{nn}^{(1)} + \FF_{mm}^{(2)} + \FF_{nn}^{(2)}\} \pm 2 \text{Re}\{\FF_{mn}^{(1)} +\FF_{mn}^{(2)} \}$. Similarly, by applying Lemma \ref{monotonicity} with
$h = \text{Im}\{F_m \pm F_n\}, \,\text{Re}\{F_m \pm i F_n\}, \,\text{Im}\{F_m \pm iF_n\}$, we complete the proof.
\end{proof}

\subsection{Proof of matrix factorizations for a Lipschitz inclusion}\label{proof:thm:F1F2Lipschitz}

For $\epsilon>0$, we define
\begin{align}\notag%\label{covering}
\p \Omega_\epsilon :&= \left\{ \Psi(w) : |w| = \gamma_\epsilon \right\} \quad \mbox{with } \gamma_\epsilon = \gamma(1+\epsilon),
\end{align}
where $\Psi$ is given by \eqnref{conformal:Psi}. 
Since $\Psi$ is conformal in $\{w:|w|>\gamma\}$, $\Om_\epsilon$ is an analytic domain and $\Omega\subset\Om_\epsilon$. 
We now consider the scaled domain 
\beq\label{Om:delta:def}
\Om^\delta:=(1-\delta) (\Om-a_0)+a_0\Subset \Om\quad\mbox{for }\delta\in(0,1),
\eeq
where the subset relation holds due to the star-shaped condition for $\Om$. 
The exterior conformal mapping of $\Om^\delta$ is
\beq\label{conformal:Psi_delta}
\Psi^\delta(w)=w+a_0+\frac{a_1^\delta}{w}+\frac{a_2^\delta}{w^2}+\cdots \quad \mbox{for } |w|\geq \gamma^\delta
\eeq
with $\gamma^\delta=(1-\delta)\gamma$ and $a_n^\delta = (1-\delta)^{n+1} a_n$. For any $\epsilon>0$, we then set
\begin{align}\label{shrinking}
\p\Omega_{\epsilon}^\delta: &= \big\{ \Psi^\delta(w) : |w| = \gamma_\epsilon^\delta \big\} \quad \mbox{with } \gamma_\epsilon^\delta = \gamma^\delta(1+\epsilon).
\end{align}
 Note that $\Om_\epsilon$ and $\Om^\delta_\epsilon$ are simply connected analytic domains and that
$$\Om^\delta_\epsilon=(\Om^\delta)_\epsilon=(\Om_\epsilon)^\delta.$$

 For a fixed $\delta\in(0,1)$, it follows from \eqnref{Om:delta:def} that (see Figure \ref{swellshrink})
\begin{gather}\label{Omega:inclu}
\Omega_{\epsilon}^\delta \subsetneq \Omega \subsetneq \Omega_\epsilon \quad\mbox{for a sufficiently small }\epsilon>0,\\
\label{Omega:conv:ep}
 \Omega_\epsilon \downarrow \Omega \quad \mbox{and} \quad  \Omega_{\epsilon}^\delta \downarrow \Omega^\delta    \quad \mbox{as } \epsilon\to 0^+.
\end{gather}

\begin{figure}[h!]
\centering
\includegraphics[width=0.53\textwidth, trim={0 2cm 0 2cm}]{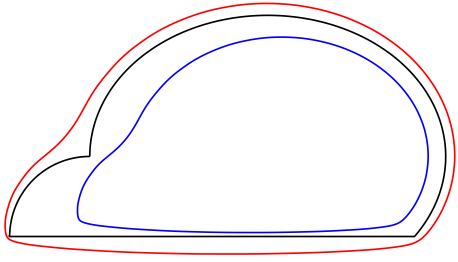}
\caption{The blue, black, and red curves indicate $\p \Omega_\epsilon^\delta$, $\p \Omega$, and $\p \Omega_\epsilon$ for a cap-shaped domain  $\Omega$, respectively, where $\Om$ is a star-shaped domain with respect to $a_0$.}
\label{swellshrink}
\end{figure}

Since $\Omega_{\epsilon}$ is an analytic domain, \eqnref{FPT:factor} holds for $\Omega_{\epsilon}$, that is,
\beq\label{eF}
\begin{aligned}
&\FF^{(1)}(\Omega_{\epsilon},\lambda) = 4\pi\, \gamma_\epsilon^\NNc \NNc^{\frac{1}{2}} \, G_\epsilon \left[ I+ (1-4\lambda^2) \left( 4\lambda^2 I - \overline{G_\epsilon} G_\epsilon \right)^{-1} \right] \NNc^{\frac{1}{2}} \,\gamma_\epsilon^\NNc, \\
&\FF^{(2)}(\Omega_{\epsilon},\lambda)= 8\pi \lambda \,\gamma_\epsilon^{\NNc} \NNc^{\frac{1}{2}} \left[ I + (1-4\lambda^2) \left( 4\lambda^2 I - \overline{G_\epsilon} G_\epsilon \right)^{-1}  \right] \NNc^{\frac{1}{2}} \,\gamma_\epsilon^{\NNc}.
\end{aligned}
\eeq
Here, $G_\epsilon$ is defined by \eqnref{g} and \eqnref{mat:G:ep} with $\gamma$ replaced by $\gamma_\epsilon$. 
In other words, 
\beq\label{Gep:expression}
G_\epsilon=(1+\epsilon)^{-\NNc}\, G\, (1+\epsilon)^{-\NNc}.
\eeq

One can easily find that rescaling and shifting of a domain do not change $G$. Namely, 
\beq\label{G_scaling_invariant}
G(\Om^\delta)=G(\Om).
\eeq
Since $G(\Om^\delta_\epsilon)$ is generated by the conformal mapping of $\Om^\delta$ with $\gamma_\epsilon^\delta=\gamma^\delta(1+\epsilon)$ instead of $\gamma^\delta$, one can easily find that $$G(\Om^\delta_\epsilon)=G(\Om_\epsilon)=G_\epsilon.$$
Therefore, from \eqnref{FPT:factor}, we have
\beq\label{deF}
\begin{aligned}
&\FF^{(1)}(\Omega_{\epsilon}^\delta,\lambda) = 4\pi\, (\gamma_\epsilon^\delta)^{\NNc} \NNc^{\frac{1}{2}} \,G_\epsilon \left[ I+ (1-4\lambda^2) \Big( 4\lambda^2 I - \overline{G_\epsilon} G_\epsilon  \Big)^{-1} \right] \NNc^{\frac{1}{2}}\, (\gamma_\epsilon^\delta)^{\NNc}, \\
&\FF^{(2)}(\Omega_{\epsilon}^\delta,\lambda)= 8\pi \lambda\, (\gamma_\epsilon^\delta)^{\NNc} \NNc^{\frac{1}{2}} \left[ I + (1-4\lambda^2) \Big( 4\lambda^2 I - \overline{G_\epsilon} G_\epsilon \Big)^{-1}  \right] \NNc^{\frac{1}{2}}\, (\gamma_\epsilon^\delta)^{\NNc},
\end{aligned}
\eeq
where the only difference from \eqnref{eF} and \eqnref{deF} is that $(\gamma_\epsilon^\delta)^{\NNc}$ is used instead of $\gamma_\epsilon^{\NNc}$.

From \eqnref{eF}, \eqnref{deF} and the fact that $\gamma^\delta_{\epsilon}=(1-\delta)\gamma_{\epsilon}$, the components of the FPTs of $\Omega_{\epsilon}^\delta$ converge to those of $\Omega_{\epsilon}$ as $\delta \to 0^+$ when $\epsilon>0$ is fixed. In other words, we have the following lemma. 
\begin{lemma}\label{lemma:F:delta}
For each fixed $\epsilon$, we have
\beq\label{lim:F:delta}
\lim_{\delta\to 0^+} \FF_{mn}^{(j)}(\Omega_{\epsilon}^\delta,\lambda)= \FF_{mn}^{(j)}(\Omega_{\epsilon},\lambda)\quad\mbox{for }j=1,2.
 \eeq
\end{lemma}

In the following, we investigate the convergence of $\FF_{mn}^{(1)}(\Omega_{\epsilon},\lambda)$ and $\FF_{mn}^{(2)}(\Omega_{\epsilon},\lambda)$ as $\epsilon$ goes to $0$ (see Proposition \ref{lim:F:ep}). This is much more difficult to prove than Lemma \ref{lemma:F:delta}. We start with a general property of a semi-infinite matrix.

Let $l^2(\CC)$ denote the vector space consisting of all complex sequences $(x_m)$ satisfying 
$\sum_{m=1}^\infty |x_m|^2<\infty$. 
We can interpret a semi-infinite matrix, namely $T=(t_{mn})$, as a linear operator from $l^2(\CC)$ to $l^2(\CC)$ given by
\begin{align}\label{G:l2}
    (x_m) &\longmapsto (y_m)\quad\mbox{with}\quad y_m = \sum_{n=1}^\infty t_{mn}x_n=\lim_{N\rightarrow\infty}\sum_{n=1}^N t_{mn}x_n,
\end{align}
assuming that the sequence of partial sums converges for each $m$ and that $(y_m)\in l^2(\CC)$. 
We denote by $\|T\|$ the operator norm of $T$ on $l^2(\CC)$, that is, 
$$
\left\lVert T \right\rVert=\left\lVert T \right\rVert_{l^2\to l^2} = \sup_{\left\lVert x \right\rVert=1} \left\lVert Tx \right\rVert.
$$
Assume that $\|T\|<\infty$. 
We put $x_n = \delta_{nk}$ to obtain
\be
\left\lVert T \right\rVert^2 = \sup_{\left\lVert (x_n) \right\rVert=1} \sum_{m=1}^\infty \left| \sum_{n=1}^\infty t_{mn} x_n \right|^2 \ge \sum_{m=1}^\infty \left| t_{mk} \right|^2.
\ee
Let $T^*$ denote the adjoint operator of $T$. Then, $T^*=(s_{mn})_{m,n=1}^\infty$ with $s_{mn}=t_{nm}$ and $\left\lVert T \right\rVert = \left\lVert T^* \right\rVert$.
By applying the above inequality to $T^*$, we have
\beq\label{eqn:l2norm:ineq}
\sum_{m=1}^\infty \left| t_{mk} \right|^2 \le \left\lVert T \right\rVert^2 \quad \mbox{and} \quad \sum_{m=1}^\infty \left| t_{km} \right|^2 \le \left\lVert T \right\rVert^2\quad\mbox{for each }k.
\eeq

The Grunsky coefficients satisfy that (see, for instance, \cite[Chapter 4.5]{Duren:1983:UF})
\begin{equation}\label{ineq:Grunsky}
\sum_{n=1}^\infty \left| \sum_{m=1}^\infty \sqrt{\frac{n}{m}} \frac{c_{mn}}{\gamma^{m+n}} x_m \right|^2 \leq \sum_{m=1}^\infty |x_m|^2
\end{equation}
for all complex sequences $(x_m)$.
From \eqnref{g} and the symmetry of $G$, we then have
\begin{align*}
\left\lVert G \right\rVert^2
& = \sup_{\left\lVert (x_n) \right\rVert=1} \sum_{m=1}^\infty \left| \sum_{n=1}^\infty g_{mn} x_n \right|^2 
= \sup_{\left\lVert (x_n) \right\rVert=1} \sum_{m=1}^\infty \left| \sum_{n=1}^\infty g_{nm} x_n \right|^2\leq 1.
\end{align*} 
Assuming that  $\p\Om$ is a piecewise smooth Jordan curve without cusps, we have (see \cite[Theorem 9.14]{Pommerenke:1975:UF})
\beq\label{gnorm}
\lVert G \rVert \leq \kappa\quad\mbox{for some }\kappa<1.
\eeq

We now consider the operator $G_\epsilon$ (see \eqnref{Gep:expression}) as follows.
\begin{lemma}\label{GGbound}
For all $\epsilon>0$ and $|\lambda|\geq\frac{1}{2}$, 
we have
$$\left\lVert G_\epsilon \right\rVert \le \left\lVert G \right\rVert$$
and 
$$\left\lVert \left( 4\lambda^2 I - \overline{G} G \right)^{-1} \right\rVert,\,
 \left\lVert \left( 4\lambda^2 I - \overline{G_\epsilon} G_\epsilon \right)^{-1} \right\rVert \le \left( 1 - \left\lVert G \right\rVert^2 \right)^{-1}.$$
\end{lemma}
\begin{proof}
From \eqnref{Gep:expression}, it is straightforward to find that
\beq\label{Gineq}
\left\lVert G_\epsilon \right\rVert = \left\lVert(1+\epsilon)^{-\NNc} G\, (1+\epsilon)^{-\NNc}\right\rVert \le \left\lVert(1+\epsilon)^{-\NNc}\right\rVert^2 \left\lVert G \right\rVert  \le \left\lVert G \right\rVert.
\eeq
From the assumption $|\lambda|\ge \frac{1}{2}$ and \eqnref{Gineq}, we derive
\begin{align*}
\left\lVert \left( 4\lambda^2 I - \overline{G_\epsilon}G_\epsilon  \right)^{-1} \right\rVert
&\le \frac{1}{4\lambda^2} \sum_{n=0}^\infty \left\lVert \frac{1}{4\lambda^2}\overline{G_\epsilon} G_\epsilon \right\rVert^n 
\le \sum_{n=0}^\infty \|G_\epsilon\|^{2n}\le \left(1-\|G \|^{2}\right)^{-1}.
\end{align*}
Similarly, we have boundedness for $\left(4\lambda^2 I - \overline{G} G \right)^{-1}$ .
\end{proof}

The following lemma is essential in proving the convergence of $\FF_{mn}^{(1)}(\Omega_{\epsilon},\lambda)$ and $\FF_{mn}^{(2)}(\Omega_{\epsilon},\lambda)$ as $\epsilon$ tends to zero. 
\begin{lemma}\label{GepsilonG_bound}
Let $X=\left(x_{mn}\right)_{m,n=1}^\infty$ and $Y=\left(y_{mn}\right)_{m,n=1}^\infty$ be semi-infinite matrices depending on $\epsilon$. Assume that $X$, $Y$ are uniformly bounded with respect to $\epsilon$, i.e., $\left\lVert X \right\rVert, \left\lVert Y \right\rVert \le M<\infty$ for some constant $M$ independent of $\epsilon$. Then, 
the $(m,n)$-component of $X(G_\epsilon - G)Y$ satisfies that 
\be
\lim_{\epsilon \to 0^+}\big[ X (G_\epsilon - G) Y \big]_{mn} = 0\quad\mbox{for each }m,n.
\ee
\end{lemma}
\begin{proof}
Fix $m,n\in\NN$. 
We have
\begin{align*}
&\Big|\big[ X (G_\epsilon - G) Y \big]_{mn}\Big|\\
= &\left|\sum_{k=1}^\infty \sum_{l=1}^\infty x_{mk} \, g_{kl} \Big[(1+\epsilon)^{-k-l} - 1 \Big] y_{ln} \right| \\
=&\left|\sum_{k=1}^\infty \sum_{l=1}^\infty x_{mk} \, g_{kl} \Big[ \left((1+\epsilon)^{-k} - 1\right)(1+\epsilon)^{-l}+ \left( (1+\epsilon)^{-l} - 1\right) \Big] y_{ln} \right| \\
\le & \sum_{k=1}^\infty \left| \left(  (1+\epsilon)^{-k} - 1 \right)x_{mk} \sum_{l=1}^\infty g_{kl} (1+\epsilon)^{-l} y_{ln}\right| 
 + \sum_{k=1}^\infty \left|x_{mk}\sum_{l=1}^\infty \, g_{kl} \left(  (1+\epsilon)^{-l} - 1 \right) y_{ln} \right| \\
=&: S_1 + S_2.
\end{align*}
It is sufficient to show that $S_1, S_2 \to 0$ as $\epsilon\to0^+$.

From the Cauchy-Schwarz inequality and \eqnref{ineq:Grunsky}, we have
\begin{align}
S_1
&\le \left( \sum_{k=1}^\infty \Big|\left(  (1+\epsilon)^{-k} - 1 \right)x_{mk} \Big|^2 \right)^{\frac{1}{2}}
\left( \sum_{k=1}^\infty \left|\sum_{l=1}^\infty g_{kl} (1+\epsilon)^{-l} y_{ln}\right|^2 \right)^{\frac{1}{2}} \notag\\
&\le \left( \sum_{k=1}^\infty \big( 1 - (1+\epsilon)^{-k} \big)^2 \left|x_{mk} \right|^2 \right)^{\frac{1}{2}}\left( \sum_{l=1}^\infty \Big|(1+\epsilon)^{-l} y_{ln}\Big|^2 \right)^{\frac{1}{2}}.\label{S1bound}
\end{align}
It then follows from \eqnref{eqn:l2norm:ineq} that
\begin{align*}
&\sum_{k=1}^\infty \left( 1 - (1+\epsilon)^{-k} \right)^2 \left|x_{mk} \right|^2
\le \sum_{k=1}^\infty \left|x_{mk} \right|^2
\le \left\lVert X \right\rVert^2 \leq M^2,\\
&\sum_{l=1}^\infty \Big|(1+\epsilon)^{-l} y_{ln}\Big|^2
\leq 
\sum_{l=1}^\infty \left|y_{ln}\right|^2 \leq \|Y\|^2\leq M^2\quad\mbox{independent of }\epsilon. 
\end{align*}
Applying the dominated convergence theorem, we obtain
\beq\label{dct}
\lim_{\epsilon\to0^+} \sum_{k=1}^\infty \left( 1 - (1+\epsilon)^{-k} \right)^2 \left|x_{mk} \right|^2 = \sum_{k=1}^\infty \lim_{\epsilon\to0^+}  \left( 1 - (1+\epsilon)^{-k} \right)^2 \left|x_{mk} \right|^2 = 0.
\eeq
From \eqnref{S1bound} and \eqnref{dct}, $S_1$ converges to $0$ as $\epsilon\to0^+$.

Similarly, we derive
\begin{align}
S_2 
&= \sum_{k=1}^\infty \left|x_{mk}\sum_{l=1}^\infty g_{kl} \left(  (1+\epsilon)^{-l} - 1 \right) y_{ln} \right| \notag\\
&\le \left( \sum_{k=1}^\infty \left|x_{mk}\right|^2 \right)^{\frac{1}{2}}
\left( \sum_{k=1}^\infty \left| \sum_{l=1}^\infty g_{kl} \left(  (1+\epsilon)^{-l} - 1 \right) y_{ln} \right|^2 \right)^{\frac{1}{2}}\notag\\
&\le \left\lVert X \right\rVert  \left( \sum_{l=1}^\infty \left( 1 - (1+\epsilon)^{-l} \right)^2 \left|  y_{ln} \right|^2 \right)^{\frac{1}{2}}\rightarrow 0\quad\mbox{as }\epsilon\to 0^+.
\end{align}
This completes the proof.
\end{proof}

\begin{prop}\label{lim:F:ep}
As $\epsilon$ tends to zero, the right-hand sides of \eqnref{eF} converge to the formulas with $\gamma$ in the place of $\gamma_\epsilon$, that is, for each $m,n$,
\beq\label{F_lim_ep}
\begin{aligned}
&\lim_{\epsilon\to 0^+}\FF_{mn}^{(1)}(\Omega_{\epsilon},\lambda) = \left[4\pi\, \gamma^\NNc \NNc^{\frac{1}{2}}\, G \left[ I+ (1-4\lambda^2) \left( 4\lambda^2 I - \overline{G} G \right)^{-1} \right] \NNc^{\frac{1}{2}} \,\gamma^\NNc\right]_{mn}, \\
&\lim_{\epsilon\to 0^+}\FF_{mn}^{(2)}(\Omega_{\epsilon},\lambda) = \left[8\pi \lambda \,\gamma^{\NNc} \NNc^{\frac{1}{2}} \left[ I + (1-4\lambda^2) \left( 4\lambda^2 I - \overline{G} G \right)^{-1} \right] \NNc^{\frac{1}{2}} \,\gamma^{\NNc}\right]_{mn}.
\end{aligned}
\eeq
\end{prop}
\begin{proof}
Since $ \gamma_\epsilon^\NNc \NNc^{\frac{1}{2}} G_\epsilon \NNc^{\frac{1}{2}} \gamma_\epsilon^\NNc=C=\gamma^\NNc \NNc^{\frac{1}{2}} G \NNc^{\frac{1}{2}} \gamma^\NNc $, we cancel out the first term of $\FF_{mn}^{(1)}$ and get
\begin{align}
&\FF_{mn}^{(1)}(\Omega_{\epsilon},\lambda) - \bigg[4\pi \,\gamma^\NNc \NNc^{\frac{1}{2}}\, G \left[ I+ (1-4\lambda^2) \left( 4\lambda^2 I - \overline{G} G \right)^{-1} \right] \NNc^{\frac{1}{2}}\, \gamma^\NNc \bigg]_{mn}\notag \\ 
= &\,4\pi \sqrt{mn} \, (1-4\lambda^2)  \bigg( \gamma_\epsilon^{m+n} \left[G_\epsilon \left( 4\lambda^2 I - \overline{G_\epsilon} G_\epsilon \right)^{-1}\right]_{mn} - \gamma^{m+n} \left[G \left( 4\lambda^2 I - \overline{G} G \right)^{-1}\right]_{mn} \bigg).\notag
\end{align}
We then have
 \begin{align*}
& \gamma_\epsilon^{m+n} \left[G_\epsilon \left( 4\lambda^2 I - \overline{G_\epsilon} G_\epsilon \right)^{-1}\right]_{mn} - \gamma^{m+n} \left[G \left( 4\lambda^2 I - \overline{G} G \right)^{-1}\right]_{mn}\\
=\,& \gamma_\epsilon^{m+n} \left[ (G_\epsilon - G) \left( 4\lambda^2 I - \overline{G_\epsilon} G_\epsilon \right)^{-1} \right]_{mn}\notag\\
 + &\,\gamma_\epsilon^{m+n} \left[ G \left( 4\lambda^2 I - \overline{G_\epsilon} G_\epsilon \right)^{-1} (\overline{G_\epsilon} - \overline{G}) \, G_\epsilon \left( 4\lambda^2 I - \overline{G} G \right)^{-1} \right]_{mn}\notag\\
 +&\, \gamma_\epsilon^{m+n} \left[ G \left( 4\lambda^2 I - \overline{G_\epsilon} G_\epsilon \right)^{-1} \overline{G} \, (G_\epsilon - G) \left( 4\lambda^2 I - \overline{G} G \right)^{-1} \right]_{mn}\notag\\
+ &\, \left(\gamma_\epsilon^{m+n} - \gamma^{m+n} \right) \left[ G \left( 4\lambda^2 I - \overline{G} G \right)^{-1} \right]_{mn}.
\end{align*}
From Lemma \ref{GGbound} and Lemma \ref{GepsilonG_bound}, this term converges to $0$ as $\epsilon\rightarrow 0^+$.

Similarly, it holds that
\begin{align*}
& \FF_{mn}^{(2)}(\Omega_{\epsilon},\lambda) - \bigg[ 8\pi \lambda \,\gamma^{\NNc} \NNc^{\frac{1}{2}} \left[ I + (1-4\lambda^2) \left( 4\lambda^2 I - \overline{G} G \right)^{-1} \right] \NNc^{\frac{1}{2}}\, \gamma^{\NNc} \bigg]_{mn}\notag \\ 
=\,& 8\pi \lambda \sqrt{mn} \, \left(\gamma_\epsilon^{m+n} - \gamma^{m+n}\right)\notag \\
+ &\, 8\pi \lambda \sqrt{mn} \, (1-4\lambda^2) \bigg( \gamma_\epsilon^{m+n} \left[ \left( 4\lambda^2 I - \overline{G_\epsilon} G_\epsilon \right)^{-1}\right]_{mn} - \gamma^{m+n} \left[ \left( 4\lambda^2 I - \overline{G} G \right)^{-1}\right]_{mn} \bigg).
\end{align*}
The first term converges to $0$ as $\epsilon\to 0^+$. The second terms satisfies
\begin{align*}
& \gamma_\epsilon^{m+n} \left[ \left( 4\lambda^2 I - \overline{G_\epsilon} G_\epsilon \right)^{-1}\right]_{mn} - \gamma^{m+n} \left[ \left( 4\lambda^2 I - \overline{G} G \right)^{-1}\right]_{mn} \\
=\,&  \left( \gamma_\epsilon^{m+n} - \gamma^{m+n} \right) \left[ \left( 4\lambda^2 I - \overline{G} G \right)^{-1}\right]_{mn} \\ 
 + &\, \gamma_\epsilon^{m+n} \Big[ \left( 4\lambda^2 I - \overline{G_\epsilon} G_\epsilon \right)^{-1} \overline{G_\epsilon} \, (G_\epsilon - G) \left( 4\lambda^2 I - \overline{G} G \right)^{-1}\Big]_{mn} \\ 
 + &\,  \gamma_\epsilon^{m+n} \Big[ \left( 4\lambda^2 I - \overline{G_\epsilon} G_\epsilon \right)^{-1} (\overline{G_\epsilon} - \overline{G}) \, G \left( 4\lambda^2 I - \overline{G} G \right)^{-1}\Big]_{mn}.
 \end{align*}
From Lemma \ref{GGbound} and Lemma \ref{GepsilonG_bound}, this term converges to $0$ as $\epsilon\rightarrow 0^+$.
Hence, we prove the proposition.
\end{proof}

Let $\delta\in (0,1)$ be fixed.
By applying Proposition \ref{lim:F:ep} to $\Omega^\delta_\epsilon$ and using \eqnref{G_scaling_invariant},  we have
\begin{align*}
&\lim_{\epsilon\to 0^+}\FF_{mn}^{(1)}(\Omega^\delta_{\epsilon},\lambda) = \left[4\pi\, (\gamma^\delta)^\NNc \NNc^{\frac{1}{2}}\, G \left[ I+ (1-4\lambda^2) \left( 4\lambda^2 I - \overline{G} G \right)^{-1} \right] \NNc^{\frac{1}{2}} \,(\gamma^\delta)^\NNc\right]_{mn}, \\
&\lim_{\epsilon\to 0^+}\FF_{mn}^{(2)}(\Omega^\delta_{\epsilon},\lambda) =\left[ 8\pi \lambda \,(\gamma^\delta)^{\NNc} \NNc^{\frac{1}{2}} \left[ I + (1-4\lambda^2) \left( 4\lambda^2 I - \overline{G} G \right)^{-1} \right] \NNc^{\frac{1}{2}} \,(\gamma^\delta)^{\NNc}\right]_{mn}.
\end{align*}
Because of $\gamma^\delta=(1-\delta)\gamma$, one can easily find that $\lim_{\epsilon\to 0^+}\FF_{mn}^{(j)}(\Omega^\delta_{\epsilon},\lambda)$ converges to the right-hand sides of the equations in \eqnref{F_lim_ep} as $\delta\to0^+$. 
In view of Lemma \ref{lemma:F:delta} and Proposition \ref{lim:F:ep}, we conclude that
\beq\label{limit_order}
\lim_{\epsilon\to 0^+}\lim_{\delta\to 0^+}\FF_{mn}^{(j)}(\Omega^\delta_{\epsilon},\lambda) 
=\lim_{\epsilon\to 0^+}\FF_{mn}^{(j)}(\Omega_{\epsilon},\lambda) 
=\lim_{\delta\to 0^+}\lim_{\epsilon\to 0^+}\FF_{mn}^{(j)}(\Omega^\delta_{\epsilon},\lambda)\quad\mbox{for }j=1,2.
\eeq

\smallskip

\noindent{\textbf{Proof of Theorem \ref{conformal2GPT}}.}
From \eqnref{FPNP} and the fact that $P$ is invertible, it is sufficient to prove that the FPTs of $\Om$ satisfy \eqnref{FPT:factor} under the same assumptions as in Theorem \ref{conformal2GPT}. 

Fix indices $m,n$. Assume that $\lambda>\frac{1}{2}$. As in Lemma \ref{diagonalmonotone}, we set
$$\mathbb{A}_{mn}^{(+1)}=\text{Re}\left\{ \FF_{mm}^{(1)} + \FF_{nn}^{(1)} + \FF_{mm}^{(2)} + \FF_{nn}^{(2)}\right\} + 2 \text{Re}\left\{\FF_{mn}^{(1)} +\FF_{mn}^{(2)} \right\}.$$
From \eqnref{limit_order}, it holds that
\beq\label{AA_inequa}
\lim_{\epsilon\to0^+} \mathbb{A}_{mn}^{(1)}(\Omega_\epsilon, \lambda)
=\lim_{\delta\to0^+}\lim_{\epsilon\to0^+} \mathbb{A}_{mn}^{(1)}(\Omega_\epsilon^\delta, \lambda).
\eeq
From Lemma \ref{diagonalmonotone} and the fact that $\Omega\subsetneq\Omega_\epsilon$ for $\epsilon>0$, we have
\beq\label{I:eqn1}
\begin{aligned}
\lim_{\epsilon\to0^+} \mathbb{A}_{mn}^{(1)}(\Omega_\epsilon, \lambda)
\geq& \mathbb{A}_{mn}^{(1)}(\Omega, \lambda)
\end{aligned}
\eeq
On the other hand, \eqnref{Omega:inclu} and the monotonicity of $\mathbb{A}_{mn}^{(1)}$ imply that
$$\lim_{\epsilon\to0^+} \mathbb{A}_{mn}^{(1)}(\Omega_\epsilon^\delta, \lambda)\leq  \mathbb{A}_{mn}^{(1)}(\Omega, \lambda) \quad \mbox{for each fixed }\delta\in(0,1).$$
Therefore, we derive
$$\lim_{\delta\to0^+}\lim_{\epsilon\to0^+} \mathbb{A}_{mn}^{(1)}(\Omega_\epsilon^\delta, \lambda)\leq  \mathbb{A}_{mn}^{(1)}(\Omega, \lambda).$$
By applying this relation to \eqnref{AA_inequa} and \eqnref{I:eqn1}, we conclude that
\begin{equation}\notag%\label{I:limit}
\begin{aligned}
\mathbb{A}_{mn}^{(1)}(\Omega, \lambda)=&\lim_{\epsilon\to 0^+} \mathbb{A}_{mn}^{(1)}(\Omega_\epsilon, \lambda).
\end{aligned}
\end{equation}
One can also prove this relation for $\lambda<-\frac{1}{2}$. 

Furthermore, similar relations hold for $\mathbb{A}_{mn}^{(-1)}$ and $\mathbb{A}_{mn}^{(\pm j)}$, $j=2,3,4$. Note that $\FF_{nm}^{(1)}$ and $\FF_{nm}^{(2)}$ are linear combinations of $\mathbb{A}_{mn}^{(\pm j)}$, $j=1,2,3,4$. 
It then directly follows that, for each $m,n$,
$$\FF_{mn}^{(k)}(\Omega, \lambda)=\lim_{\epsilon\to 0^+} \FF_{mn}^{(k)}(\Omega_\epsilon, \lambda),\ k=1,2.
$$
Combining this convergence and Proposition \ref{lim:F:ep}, we conclude that the factorization \eqnref{FPT:factor} also holds for $\Om$ that satisfies the assumptions in Theorem \ref{conformal2GPT}.
\qed

\section{Numerical results}\label{sec:numerical}

In this section, we propose a semi-analytic imaging scheme for a planar conductivity inclusion with arbitrary constant conductivity based on Theorems \ref{conformalrecovery} and \ref{thm:recovery:Lip}. To demonstrate the validity of the proposed reconstruction approach, we present numerical simulations with objects of different shapes.

\subsection{Reconstruction scheme}\label{subsec:scheme}

Let $\Om$ be an unknown inclusion having constant conductivity $\sigma_c$. Set $\lambda=\frac{\sigma_c+\sigma_m}{2(\sigma_c-\sigma_m)}$.
For an inclusion with a $C^{1,\alpha}$ boundary, we apply Theorem \ref{conformalrecovery} with measurements of ${\NN}^{(1)}$ and $\NN^{(2)}$. 
By Theorem \ref{thm:recovery:Lip}, the same reconstruction procedure is valid for inclusions with corners. 

We first find $\lambda$ by using the constraint equation \eqnref{lambda:GPTs} in Theorem \ref{conformalrecovery}\,(a), that is, 
\beq\label{lambda_eqn}
\lambda  = \pi\,\frac{ \widetilde{\NN}^{(2)}_{11} \, \widetilde{\NN}^{(2)}_{22} -  \widetilde{\NN}^{(2)}_{12} \, \widetilde{\NN}^{(2)}_{21} }{\left( \widetilde{\NN}^{(2)}_{11}\right)^3}.
\eeq
We note that the modified GPTs $\widetilde{\NN}^{(1)}(\Om,\lambda)$ and $\widetilde{\NN}^{(2)}(\Om,\lambda)$ are defined by \eqnref{def:tildeNN} in terms of $\lambda$ and the original contracted GPTs  ${\NN}^{(1)}(\Om,\lambda)$ and ${\NN}^{(2)}(\Om,\lambda)$.
We can rewrite \eqnref{lambda_eqn} as 
$$\lambda = f(\lambda),\quad f(t):= \pi\,\frac{ \widetilde{\NN}^{(2)}_{11}(t) \, \widetilde{\NN}^{(2)}_{22}(t) -  \widetilde{\NN}^{(2)}_{12}(t) \, \widetilde{\NN}^{(2)}_{21}(t) }{\left( \widetilde{\NN}^{(2)}_{11}(t)\right)^3}
$$
with a variable $t\in(-1/2, 1/2)$ and
\begin{gather*}
\widetilde{\NN}^{(1)}(t):= \mathbb{N}^{(1/2)} \,\MM(t)\, \NN^{(2)},\quad
 \widetilde{\NN}^{(2)}(t):=\MM(t)\, \NN^{(2)},\\
\MM(t)=\left(I-\overline{\mathbb{N}^{(1/2)} }\,\mathbb{N}^{(1/2)}\right)\left(I-4t^2\,\overline{\mathbb{N}^{(1/2)} }\,\mathbb{N}^{(1/2)}\right)^{-1},
\end{gather*}
where $\mathbb{N}^{(1)}=\mathbb{N}^{(1)}(\Om,\lambda)$, $\mathbb{N}^{(2)}=\mathbb{N}^{(2)}(\Om,\lambda)$ and $\mathbb{N}^{(1/2)}=\mathbb{N}^{(1)} \big(\mathbb{N}^{(2)} \big)^{-1}$ are given by measurements and are not modified by $t$.
Since it is not possible to explicitly solve \eqnref{lambda_eqn} for $\lambda$, we instead apply the fixed-point iteration method to find the numerical solution. 
Afterward, we retrieve the conformal mapping coefficients $\gamma$, $a_0$, and $a_m$ by the explicit formula \eqnref{conformal:GPTs} in Theorem \ref{conformalrecovery}\,(b).

To develop the recovery method to be more realistic, we replace the semi-infinite matrices of the contacted GPTs by their finite section matrices. In other words, for some $\textrm{Ord}\geq 2$, we approximate $\NN^{(1)}$, $\NN^{(2)}$ by the truncated matrices as
\beq\label{ordGPTs}
\NN^{(j)} \approx \left(\NN_{mn}^{(j)}\right)_{1\leq m,n\leq \textrm{Ord}}, \quad j=1,2,
\eeq
and compute
$\mathbb{N}^{(1/2)}=\mathbb{N}^{(1)} \big(\mathbb{N}^{(2)} \big)^{-1}$ with the truncated matrices. With these finite approximations of $\NN^{(1)}$, $\NN^{(2)}$ and $\NN^{(1/2)}$, we recover the conductivity constant $\sigma_c$ (or, $\lambda$) and the shape of $\Om$ by the following two-step procedure.
\begin{itemize}
\item {\textbf{Step 1.}} 
Set the initial guess as
$$
\lambda_0 =  \pi\frac{\NN^{(2)}_{11}(\Om,\lambda)\NN^{(2)}_{22}(\Om,\lambda) - \NN^{(2)}_{12}(\Om,\lambda) \NN^{(2)}_{21}(\Om,\lambda) }{\left(\NN^{(2)}_{11}(\Om,\lambda) \right)^3},
$$
where the right-hand side is given by the measurements. 
For $k\geq0$, we recursively define 
\beq\notag
\lambda_{k+1}  = f(\lambda_k)
\eeq
until the tolerance criterion
$$
\left|\frac{\lambda_{k+1} - \lambda_k}{\lambda_k}\right| < 10^{-10}
$$
is met. 

\item {\textbf{Step 2.}}
We compute $\gamma$, $a_0$ and $a_m$ for $m\leq\textrm{Ord}$ by the explicit formula \eqnref{conformal:GPTs} in Theorem \ref{conformalrecovery}\,(b). With these conformal mapping coefficients, we recover the target anomaly $\Om$ by taking the image of 
$$\Psi_{\textrm{Ord}}(w)=w+a_0+\frac{a_1}{w}+\frac{a_2}{w^2}+\cdots+\frac{a_{\textrm{Ord}}}{w^{\textrm{Ord}}},\quad |w|=\gamma.$$
Then, real and imaginary parts of $\Psi_{\textrm{Ord}}(w)$ represent the coordinates of the boundary of the inclusion.
\end{itemize}

\subsection{Examples}\label{subsec:num:examples}

We show numerical results of four different shapes of the domains. Figure \ref{TargetGeometric} illustrates the shapes.  
All four domains satisfy the domain assumption in Theorem \ref{conformalrecovery} or in Theorem \ref{thm:recovery:Lip}. Examples in Figure \ref{TargetGeometric}\,(a,\,b) have $C^{1,\alpha}$ boundaries, and examples in Figure \ref{TargetGeometric}\,(c,\,d) have boundaries that are piecewise smooth Jordan curves without cusps and satisfy the star-shaped condition in Theorem \ref{thm:recovery:Lip}. We set $\sigma_m=1$. 
\begin{figure}[H]
	\centering
	\subfloat[\label{kite}Kite]{\includegraphics[width=0.24\textwidth, trim = 1cm 0 1.5cm 0]{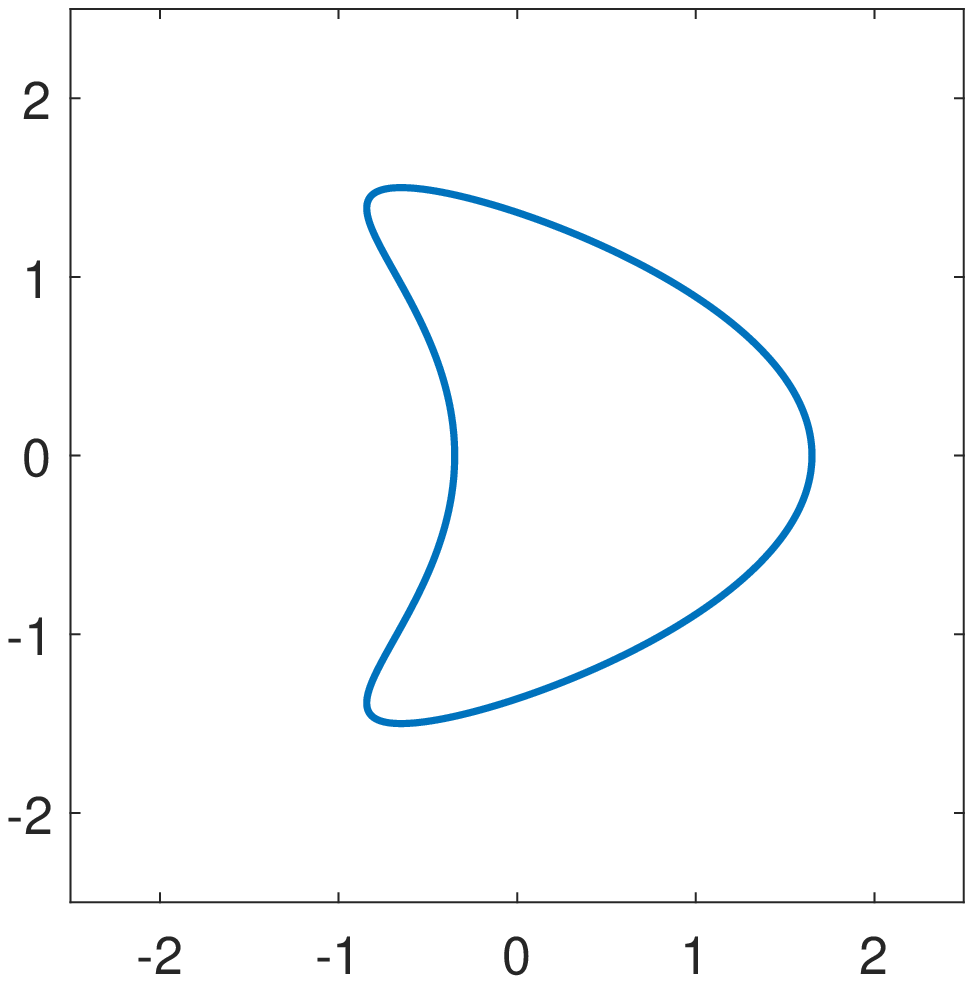}}
	\subfloat[\label{star}Starfish]{\includegraphics[width=0.24\textwidth, trim = 1cm 0 1.5cm 0]{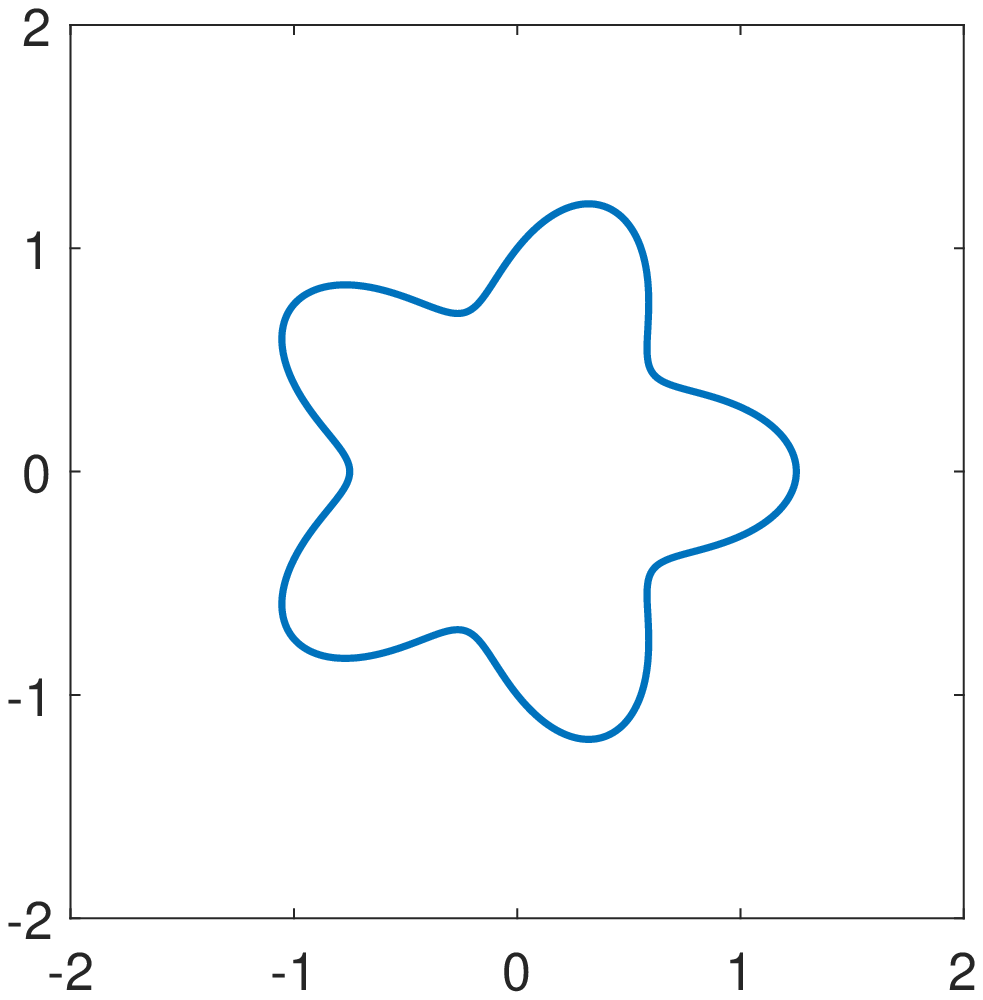}}
	\subfloat[\label{cap}Cap]{\includegraphics[width=0.24\textwidth, trim = 1cm 0 1.5cm 0]{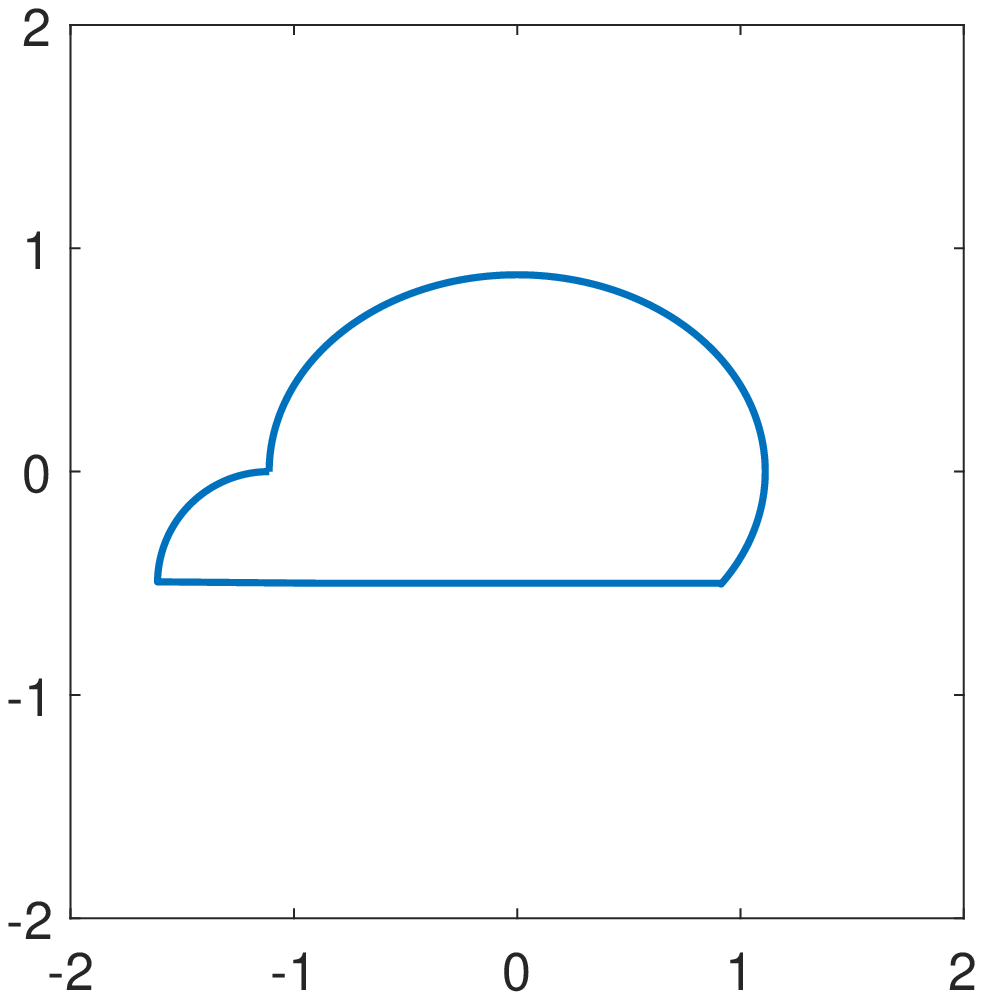}}
	\subfloat[\label{ellipse}Perturbed ellipse]{\includegraphics[width=0.24\textwidth, trim = 1cm 0 1.5cm 0]{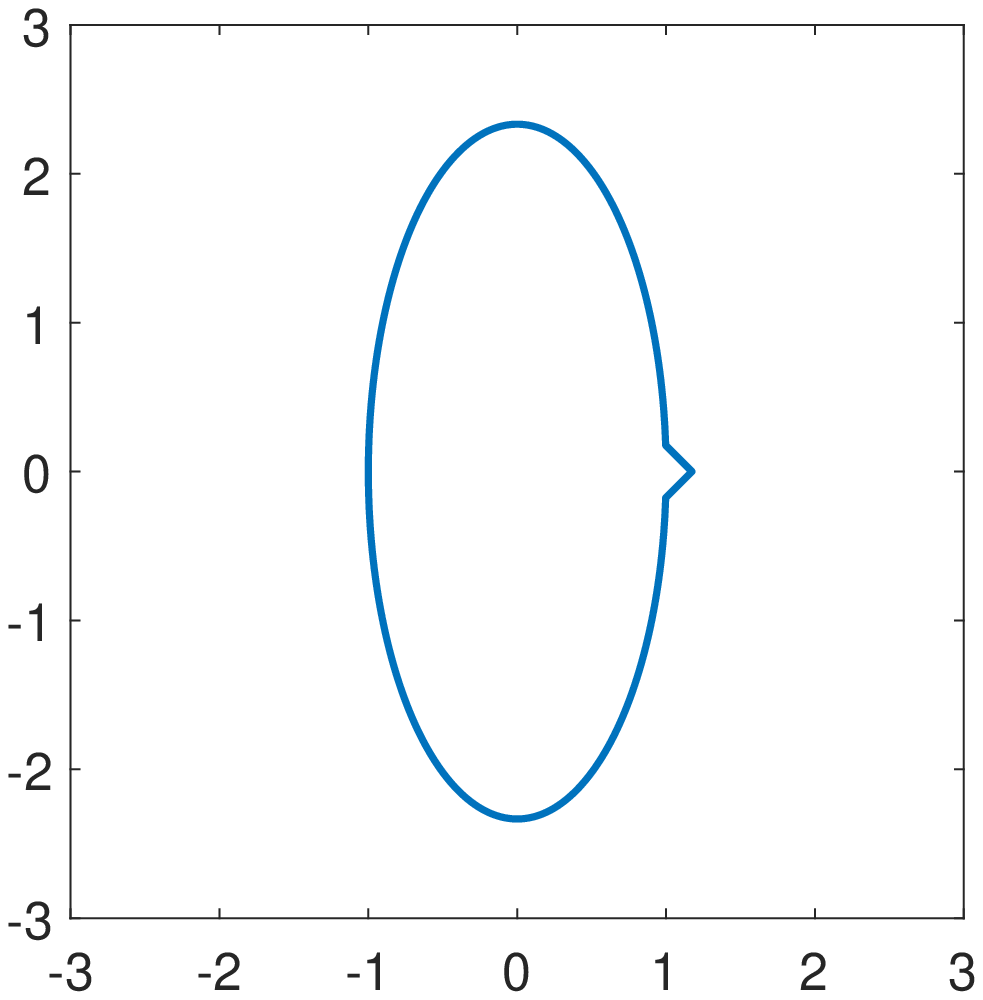}}
	\caption{Target geometries of simply connected planar inclusions. The kite- and starfish-shaped domains have $C^{1,\alpha}$ boundaries, and the cap-shaped domain and perturbed ellipse are Lipschitz domains satisfying the star-shaped condition in Theorem \ref{thm:recovery:Lip}. }
	\label{TargetGeometric}
\end{figure}

 To obtain numerical values of the GPTs, we compute the integral definition \eqnref{def:NN} for the GPTs, where the definition involves a Fredholm second-kind integral equation containing the NP operator, by applying Nystr\"{o}m discretization to this integral equation.
See \cite[Sections 17.1, 17.3]{Ammari:2013:MSM} for numerical codes to compute the GPTs of smooth domains. In our examples with corners, the Nystr\"{o}m discretization is accelerated and stabilized using {\it recursively compressed inverse preconditioning} (RCIP) \cite{Helsing:2013:SIE}. We refer the reader to \cite{Choi:2018:CEP, Helsing:2022:SFS, Helsing:2017:CSN} for computational examples of this procedure.

\begin{example}[Kite-shaped domain]\label{Ex1} \rm
First, we consider the kite-shaped inclusion whose shape is portrayed in Figure \ref{TargetGeometric}\,(a). The boundary curve is parametrized by
\begin{equation*}
x(t) = \big(\cos(t)+0.65\cos(2t),1.5\sin(t) \big),\quad t\in[0,2\pi).
\end{equation*}
Figure \ref{Ex1-Result} presents reconstruction results for the material parameter $\sigma_c$ (equivalently, $\lambda=\frac{\sigma_c+\sigma_m}{2(\sigma_c-\sigma_m)}$) and shape of the inclusion from the GPTs $\NN^{(j)}_{mn}$, $m,n\leq \textrm{Ord}$, with various $\textrm{Ord}$.
For this and the following examples, the black solid curve indicates the boundary of the target domain, the red dotted curve indicates the recovered boundary, and $\lambda^{\mbox{rec}}$ denotes the reconstructed value of $\lambda$. The imaging accuracy improves as Ord increases. 
Even with low orders of the GPTs, the location and shape of the target are approximately identified (see Figure \ref{Ex1-Result}\,(a,b)). The shape recovery scheme precisely produces the shape of the target using higher-order GPTs, as exhibited in Figure \ref{Ex1-Result}\,(d).
\end{example}

%% figures
\begin{figure}[H]
	\captionsetup[subfloat]{margin=8pt,format=hang,singlelinecheck=false}
\centering
\subfloat[\label{Ex1-1}Ord$=2$,\newline $\lambda^{\mbox{rec}}\approx1.0751$]{\includegraphics[width=0.24\textwidth]{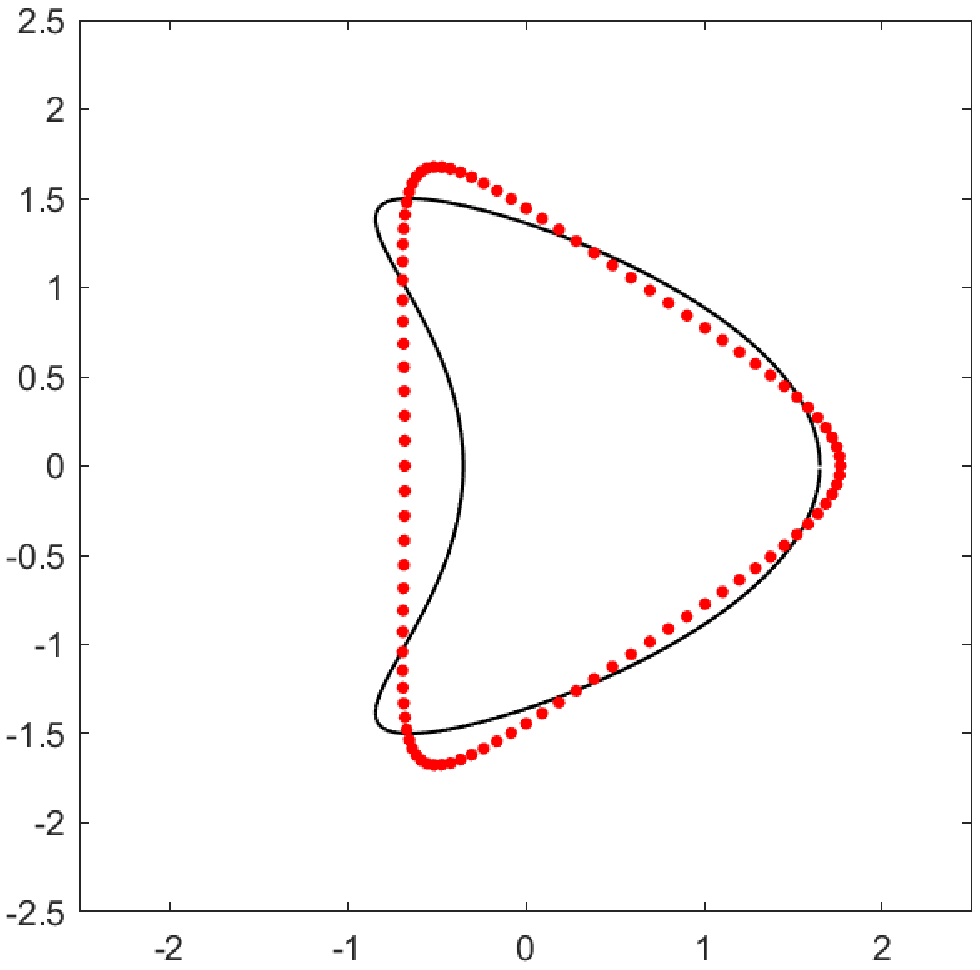}}
\subfloat[\label{Ex1-2}Ord$=3$,\newline $\lambda^{\mbox{rec}}\approx1.0246$]{\includegraphics[width=0.24\textwidth]{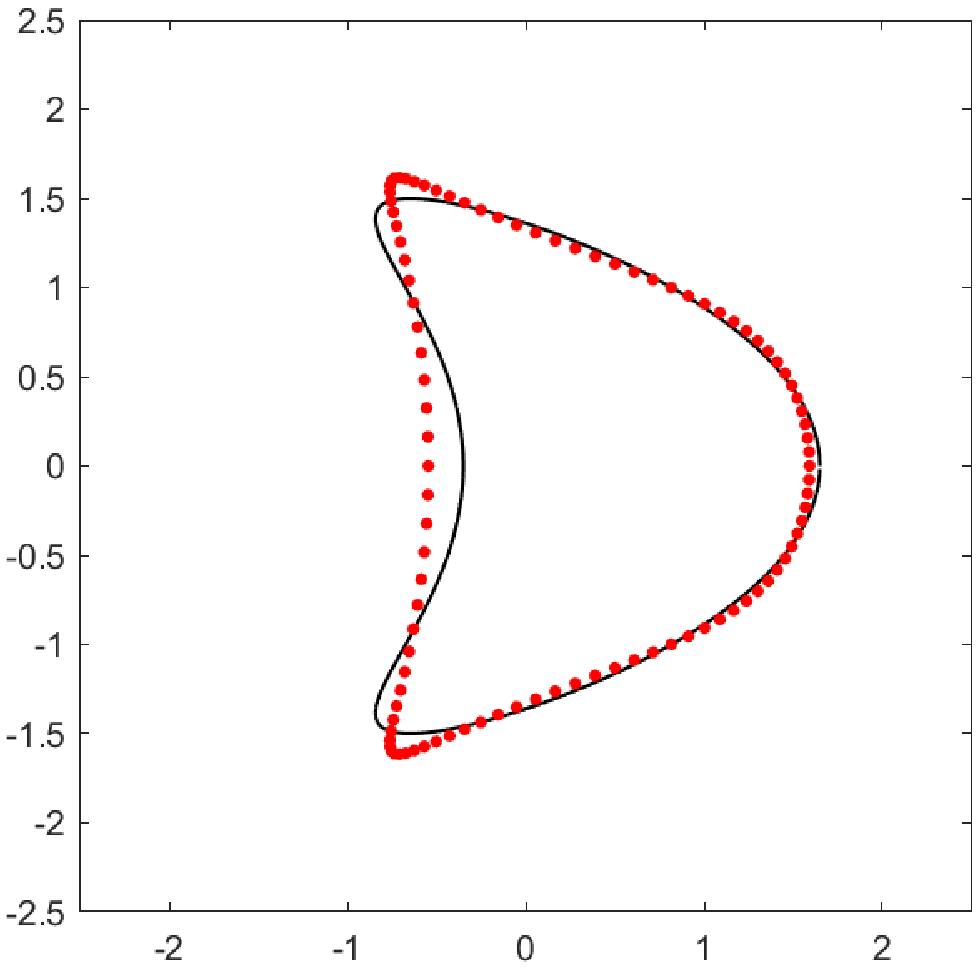}}
\subfloat[\label{Ex1-3}Ord$=5$, \newline $\lambda^{\mbox{rec}}\approx1.0036$]{\includegraphics[width=0.24\textwidth]{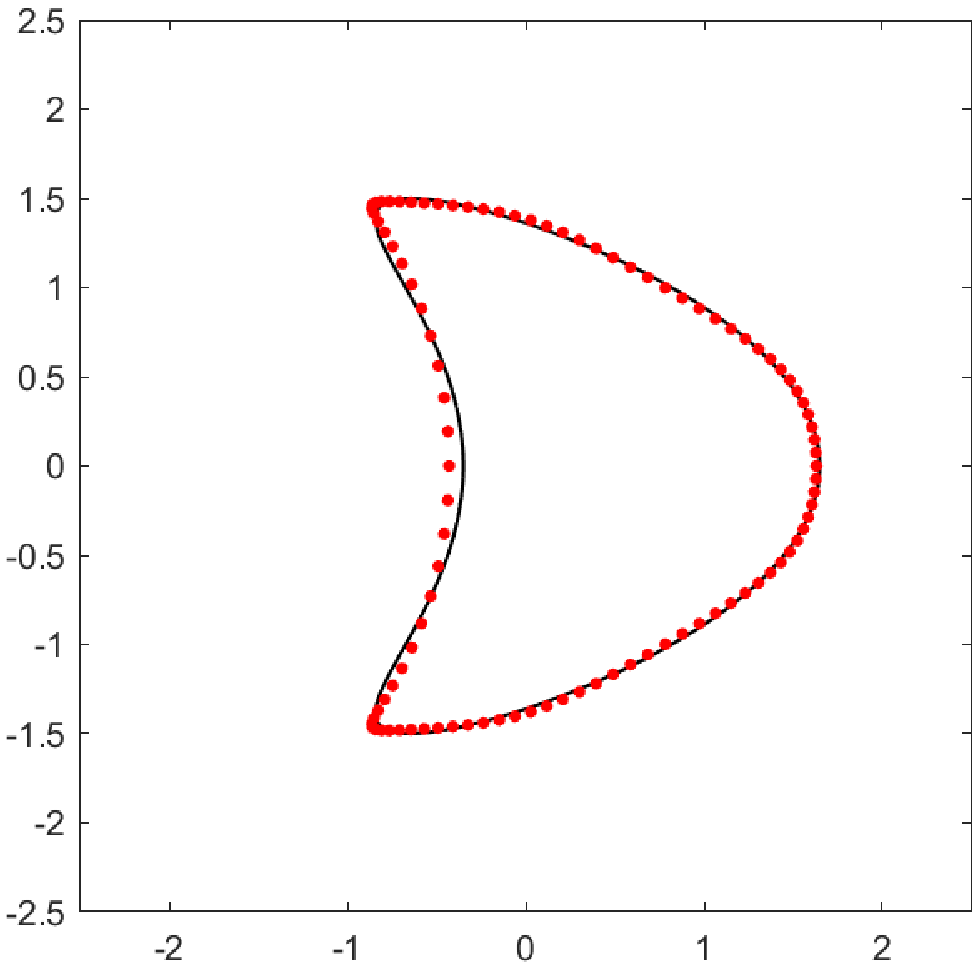}}
\subfloat[\label{Ex1-4}Ord$=10$,\newline
$\lambda^{\mbox{rec}}\approx 1.0000$]{\includegraphics[width=0.24\textwidth]{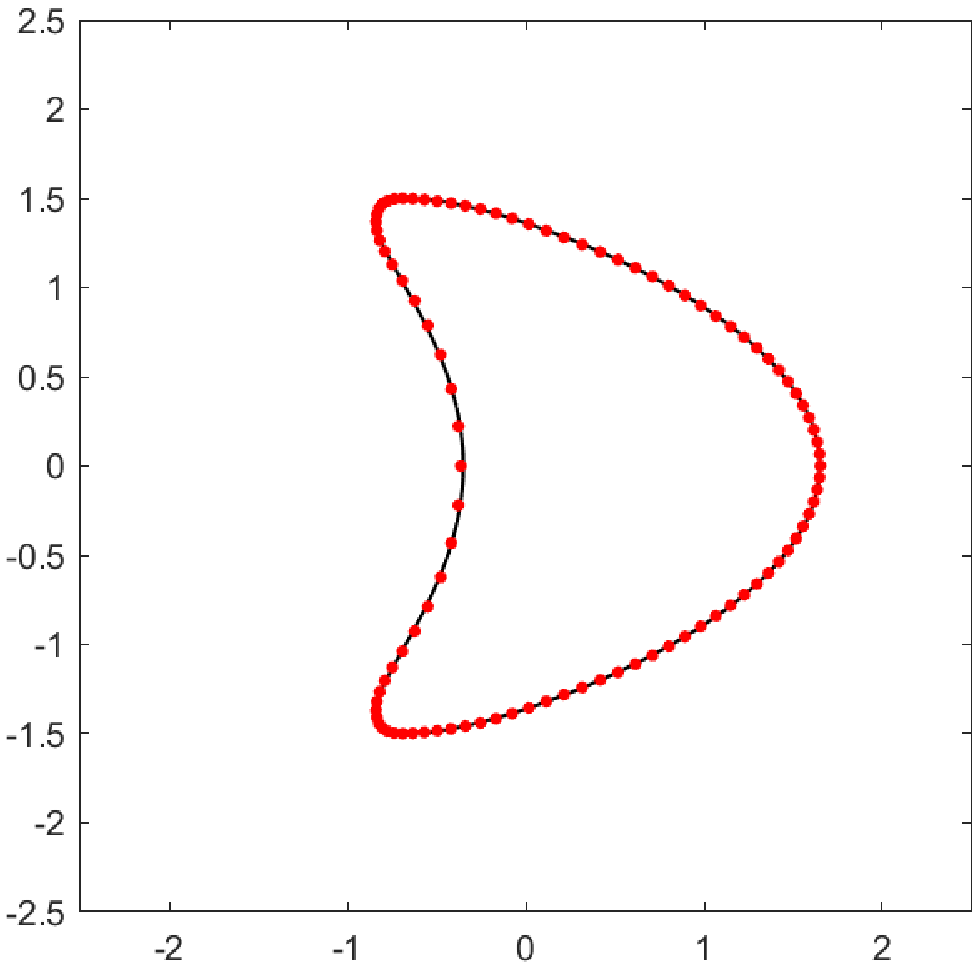}}
\caption{Recovery of the kite-shaped domain with $\sigma_c = 3$ (i.e., $\lambda = 1$). The GPTs $\NN^{(1)}_{mn}$ and $\NN^{(2)}_{mn}$ are used for $m,n\leq\textrm{Ord}=2,3,5,10$. The black solid curve indicates the boundary of the target domain, the red dotted curve indicates the recovered boundary, and $\lambda^{\mbox{rec}}$ denotes the reconstructed value of $\lambda$.}
\label{Ex1-Result}
\end{figure}

\begin{example}[Starfish-shaped domain]\label{Ex2} \rm
In this example, the starfish-shaped domain displayed in Figure \ref{TargetGeometric}\,(b) is investigated, where the parameterization is given by 
\begin{equation*}
x(t) = \big(\cos(t)+0.25\cos(5t)\cos(t), \, \sin(t)+0.25\cos(5t)\sin(t)\big),\quad t\in[0,2\pi).
\end{equation*}
For the GPTs up to various orders, Figure \ref{Ex2-Result} illustrates reconstruction results for the starfish-shaped domain. The coefficient $a_m$ of this domain decays more slowly than that of the kite-shaped domain in Example \ref{Ex1}. To approximate the boundary of the starfish-shaped domain to a high level of accuracy,  one needs the GPTs of higher orders than in Example \ref{Ex1}. Interestingly, it turns out that, in comparison to the convex parts, recovering the concave part of the boundary is more difficult.
\end{example}

We treat two objects containing corners on their boundary in Example \ref{Ex3} and Example \ref{Ex4}.

\begin{figure}[h!]
	\captionsetup[subfloat]{margin=8pt,format=hang,singlelinecheck=false}
	\centering
	\subfloat[\label{Ex2-1}Ord$=2$,\newline $\lambda^{\mbox{rec}}\approx-5.0240$]{\includegraphics[width=0.24\textwidth]{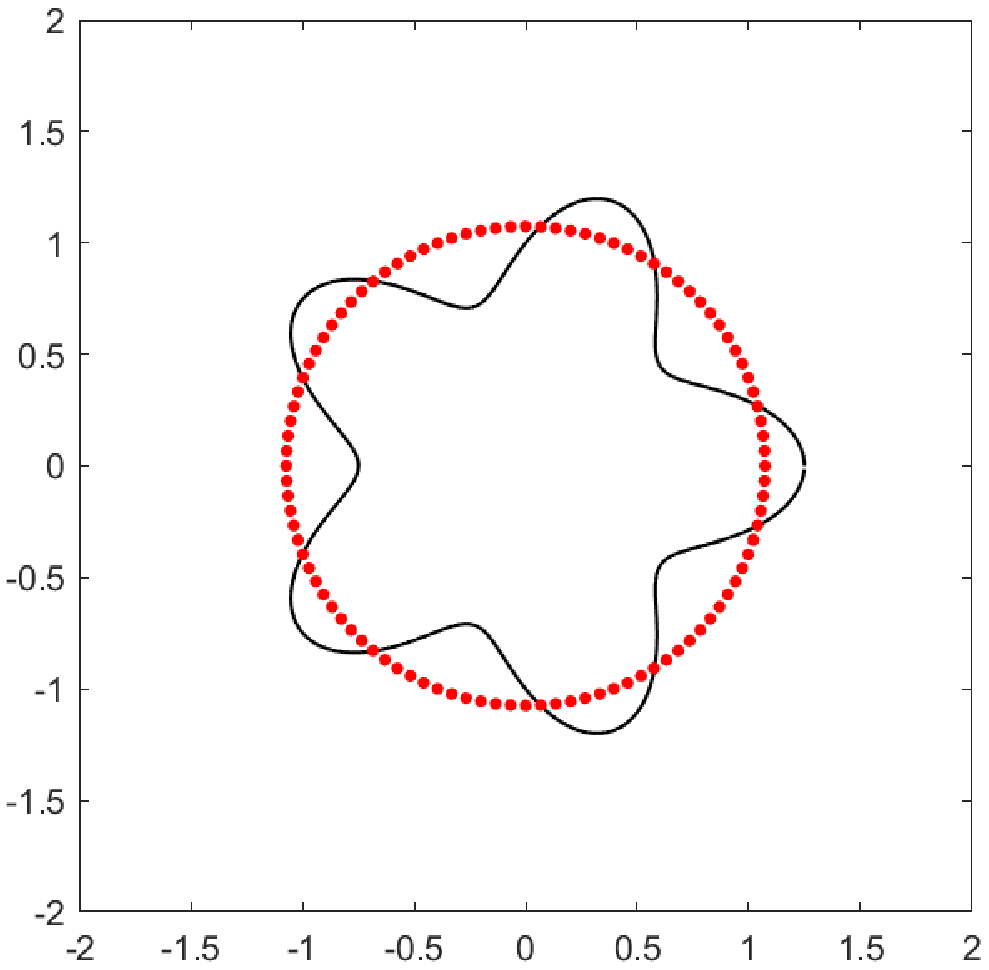}}
	\subfloat[\label{Ex2-2}Ord$=5$,\newline $\lambda^{\mbox{rec}}\approx-4.7352$]{\includegraphics[width=0.24\textwidth]{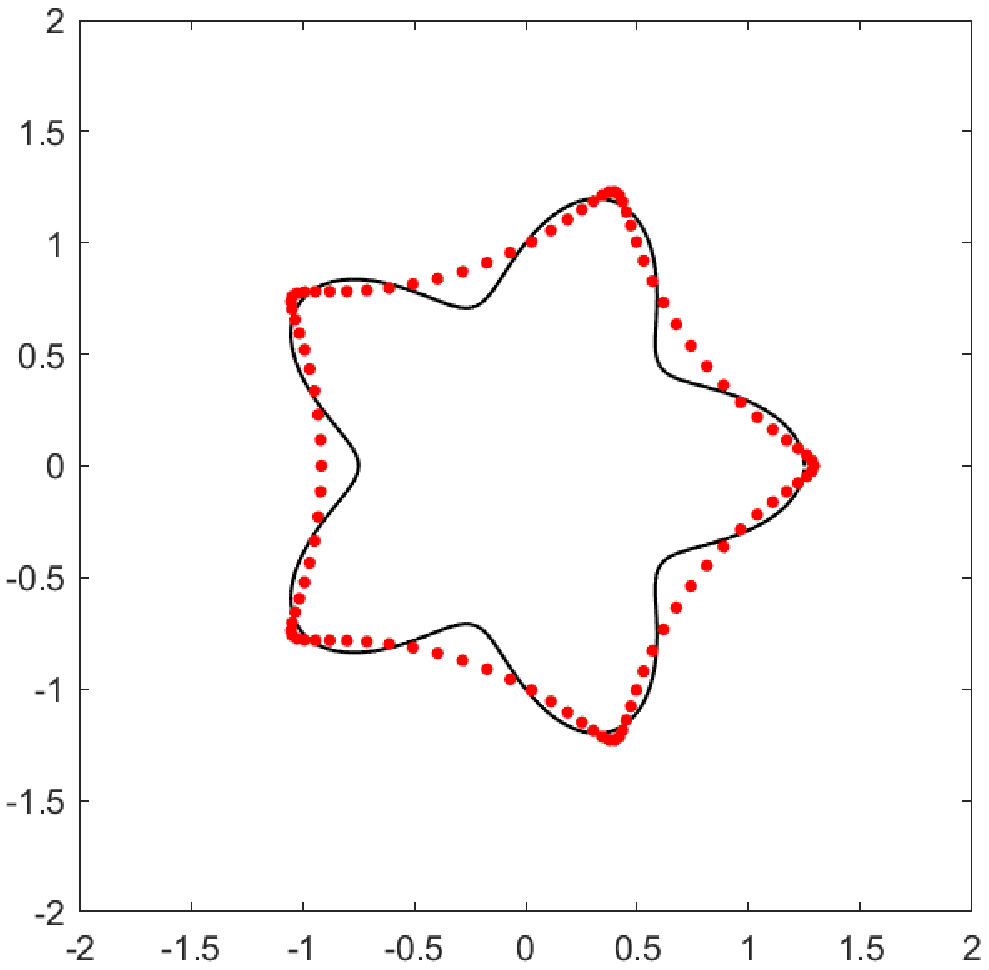}}
	\subfloat[\label{Ex2-3}Ord$=10$,\newline $\lambda^{\mbox{rec}}\approx-4.6327$]{\includegraphics[width=0.24\textwidth]{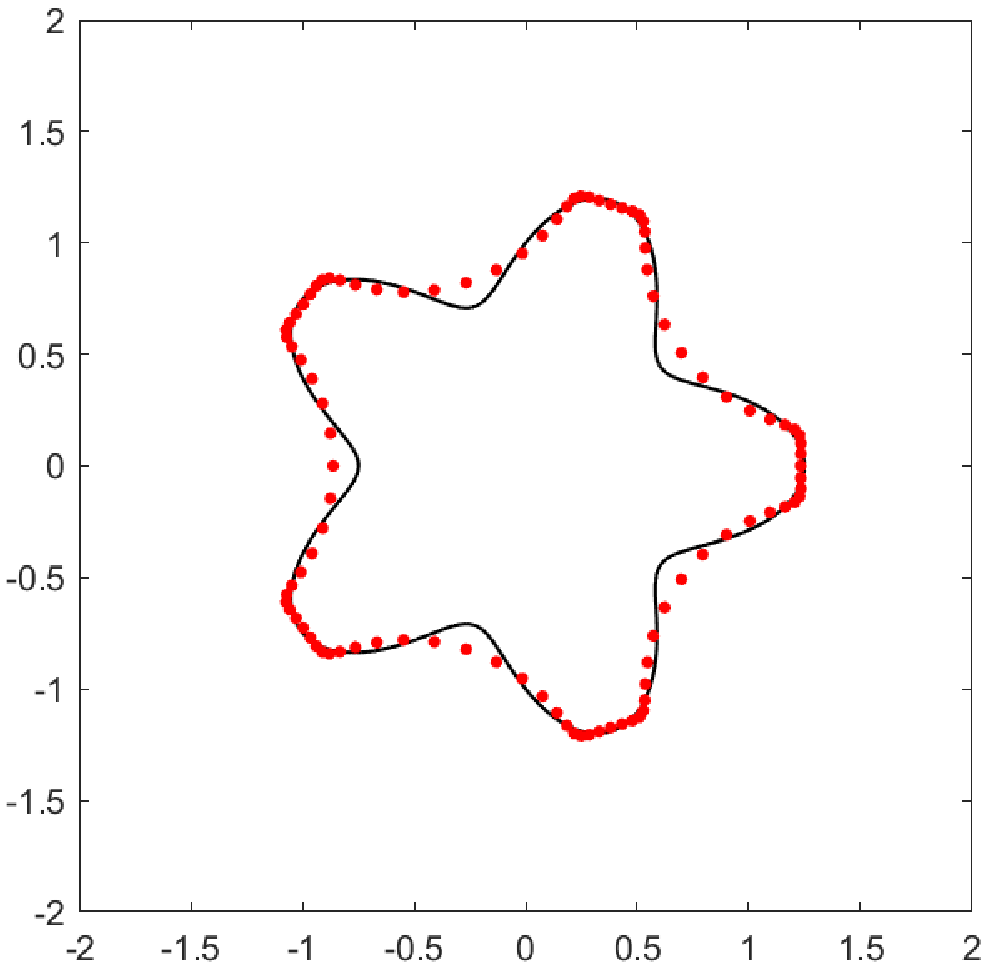}}
	\subfloat[\label{Ex2-4}Ord$=25$,\newline $\lambda^{\mbox{rec}}\approx-4.5458$]{\includegraphics[width=0.24\textwidth]{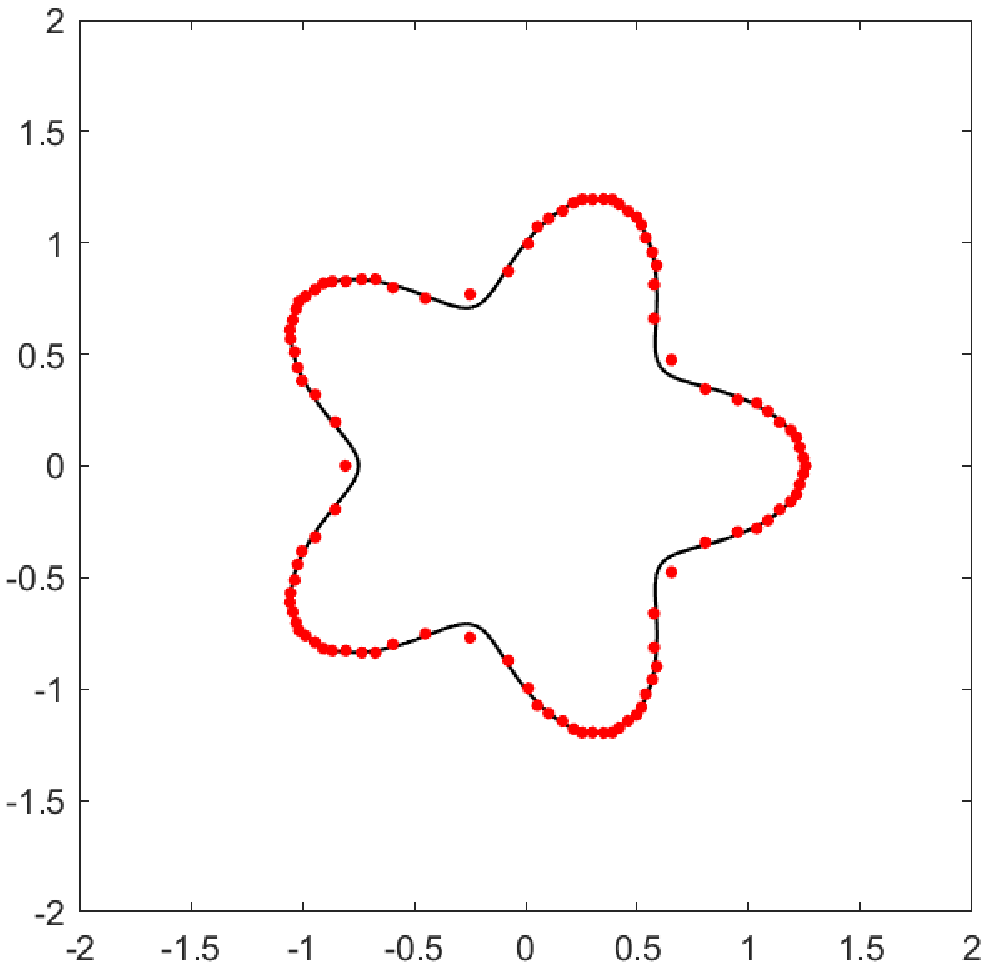}}
	%\subfloat[Order $40$]{\includegraphics[width=0.19\textwidth]{EX2-25}\label{Ex2-5}}
	\caption{Recovery of the starfish-shape domain with $\sigma= 0.8$ (i.e., $\lambda = -4.5$) from the GPTs of orders up to $\textrm{Ord}=2,5,10,25$. The coefficient $a_m$ of this domain decays in $m$ more slowly than that of the kite-shaped domain in Example \ref{Ex1}. To accurately retrieve the boundary of the domain, one needs the GPTs of higher orders than for the kite-shaped example.}
	\label{Ex2-Result}
\end{figure}
%%%

\newpage

\begin{example}[Cap-shaped domain]\label{Ex3} \rm
We consider the cap-shaped domain in Figure \ref{TargetGeometric}\,(c), which is generated by the boundary parameterization
\begin{align*}
x(t)=
\begin{cases}
\left(-\frac{1}{2}\sin\left(4\pi c t \right) - \frac{\sqrt{2}\pi}{4} , \, - \frac{1}{2} + \frac{1}{2}\cos\left(4\pi c t \right)\right),\quad &t\in[0,t_1),\\[2mm]
\left(2\pi c(t-t_2) - \sqrt{2}\arcsin\big( \frac{\sqrt{2}}{2}\cos(2\pi a) \big) , \, -\frac{1}{2} \right),\quad &t\in[t_1,t_2),\\[2mm]
\bigg(-\sqrt{2}\arcsin\left( \frac{\sqrt{2}}{2}\cos\left(2\pi c(t-t_2)+2\pi a\right) \right) , \, \\ \qquad\qquad\qquad- \arcsinh\left(\sin \left(2\pi c(t-t_2)+2\pi a\right)\right) \bigg), &t\in[t_2,1),
\end{cases}
\end{align*}
where
\begin{align*}
&a = \frac{1}{2} - \frac{1}{2\pi}\arcsin\left(\sinh\big(\frac{1}{2}\big)\right),~~
b = a-\frac{1}{4\pi}-\frac{\sqrt{2}}{8}+\frac{\sqrt{2}}{2\pi}\arcsin\left(\frac{\sqrt{2}}{2}\cos(2\pi a)\right),~~
c = \frac{9}{8}-b,\\
&t_1 = \frac{1}{8c} \approx 0.1122,\quad t_2 = t_1 + \left( \frac{a-b}{c} \right) \approx 0.4731.
\end{align*}
Figure \ref{Ex3-result} indicates that we can retrieve the conductivity $\sigma_c$ (or, $\lambda$) and shape of an inclusion even when the inclusion has corners on its boundary.  As in Example \ref{Ex2}, it is more difficult to perfectly recover the concave part of the boundary than the convex part via the proposed method.
\end{example}

\newpage

\begin{example}[Perturbed ellipse]\label{Ex4} \rm
We consider the perturbed ellipse in Figure \ref{TargetGeometric}\,(d), which has a tiny corner on its boundary and is generated by the following parameterization
\begin{align*}
x(t)=
\begin{cases}
\ds\left(\frac{a\cos(t_{0}) - c_{0}}{t_{0}}t+c_{0}, \, \frac{b\sin(t_{0})}{t_{0}}t\right),\quad&t\in[0,t_0),\\[2mm]
\ds\big(a\cos(t), \, b\sin(t)\big),\quad&t\in[t_{0},2\pi-t_0),\\[2mm]
\ds\left(\frac{a\cos(2\pi-t_{0})-c_{0}}{t_{0}}(2\pi-t)+c_{0}, \, \frac{b\sin(2\pi-t_{0})}{t_{0}}(2\pi-t)\right),\quad&t\in[2\pi-t_{0},2\pi),
\end{cases}
\end{align*}
where 
\begin{align*}
&a=1,\quad b=\frac{7}{3},\quad t_{0}=\sin^{-1}\left(\frac{1}{4b\sqrt{2}}\right)\approx0.0758,\quad
c_{0} = \sqrt{1-\frac{1}{32b^{2}}} + \frac{1}{4\sqrt{2}}\approx 1.1739.
\end{align*}
Figure \ref{Ex4-result} displays the reconstruction results with magnified images near the tiny corner. As shown in Figure \ref{Ex4-result}\,(d), the small corner can be recovered by the proposed method with high-order GPTs. We note that the corner is smoothened when $\mathrm{Ord}$ is not large enough; see Figure \ref{Ex4-result}\,(a--c).
\end{example}

%%%%%%%%%%%%%%%%%%%
% figures %

\begin{figure}[t!]
	\captionsetup[subfloat]{margin=8pt,format=hang,singlelinecheck=false}
\centering
\subfloat[\label{Ex3-1}Ord$=2$,\newline $\lambda^{\mbox{rec}}\approx-1.5107$]{\includegraphics[width=0.24\textwidth]{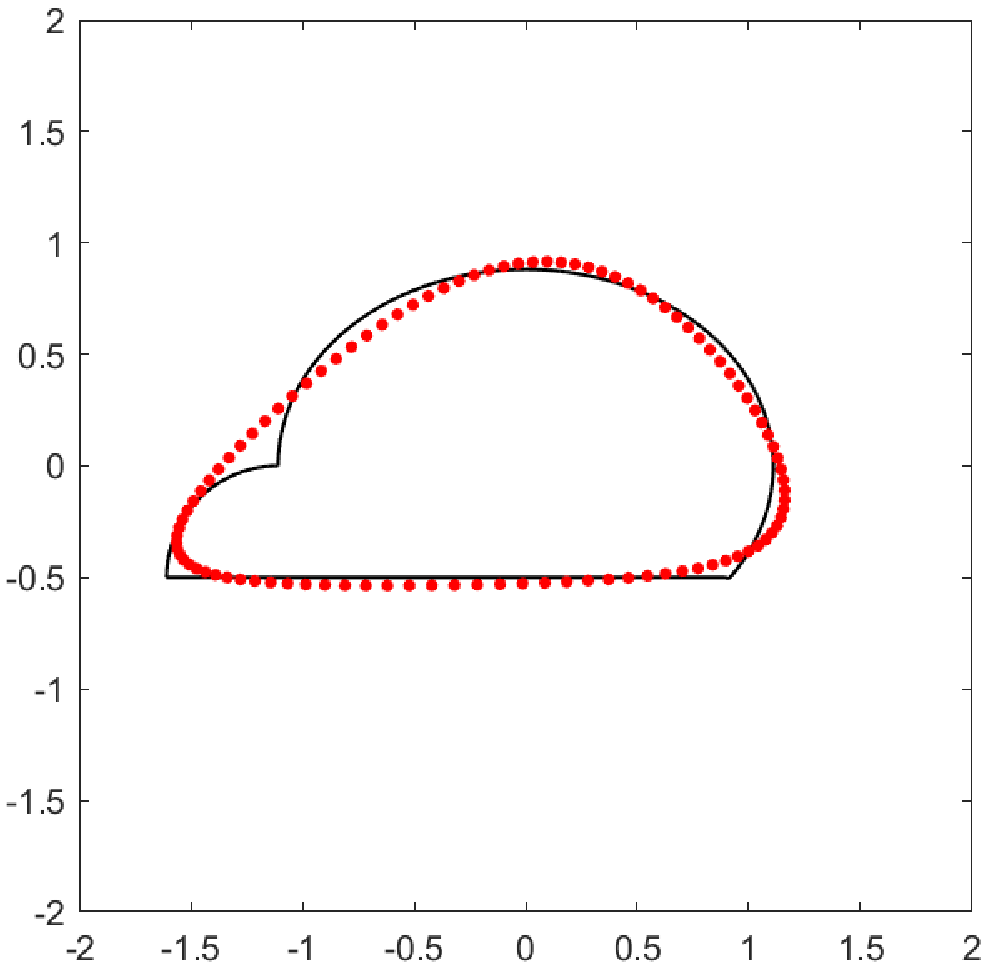}}
\subfloat[\label{Ex3-2}Ord$=5$,\newline $\lambda^{\mbox{rec}}\approx-1.5095$]{\includegraphics[width=0.24\textwidth]{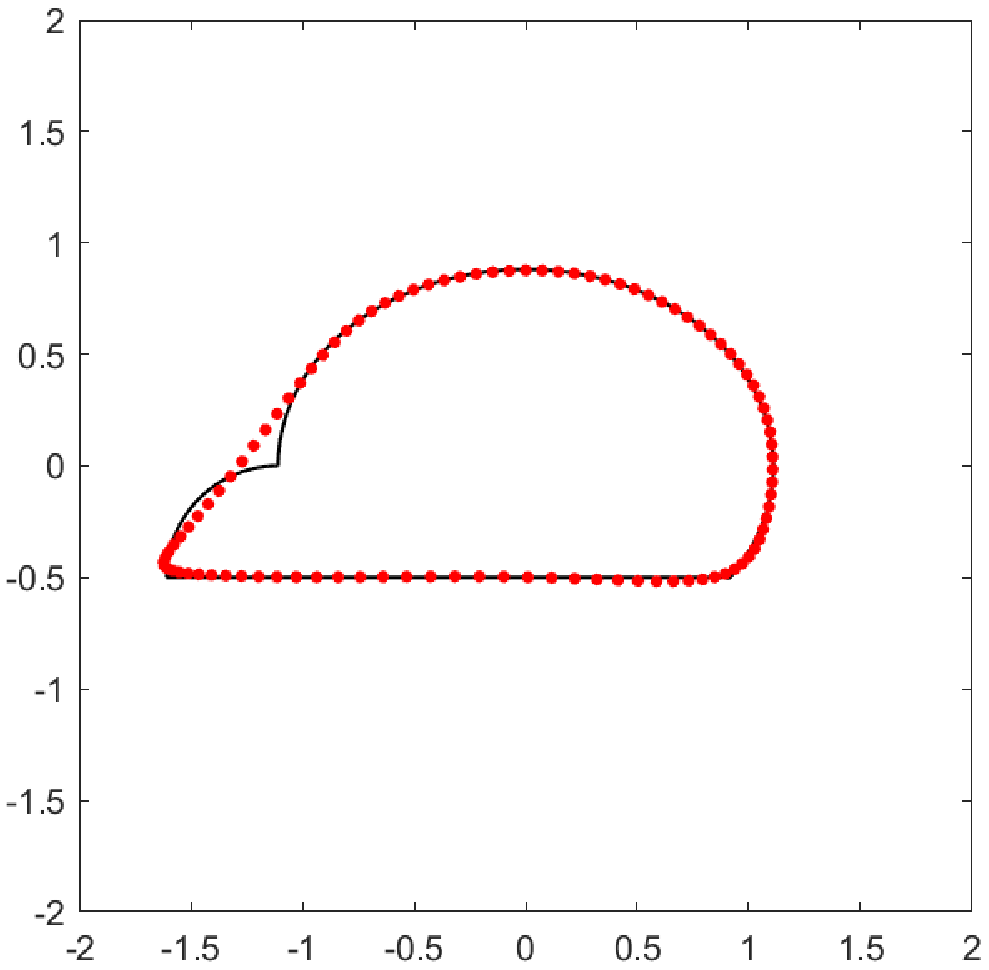}}
\subfloat[\label{Ex3-3}Ord$=10$,\newline $\lambda^{\mbox{rec}}\approx-1.5062$]{\includegraphics[width=0.24\textwidth]{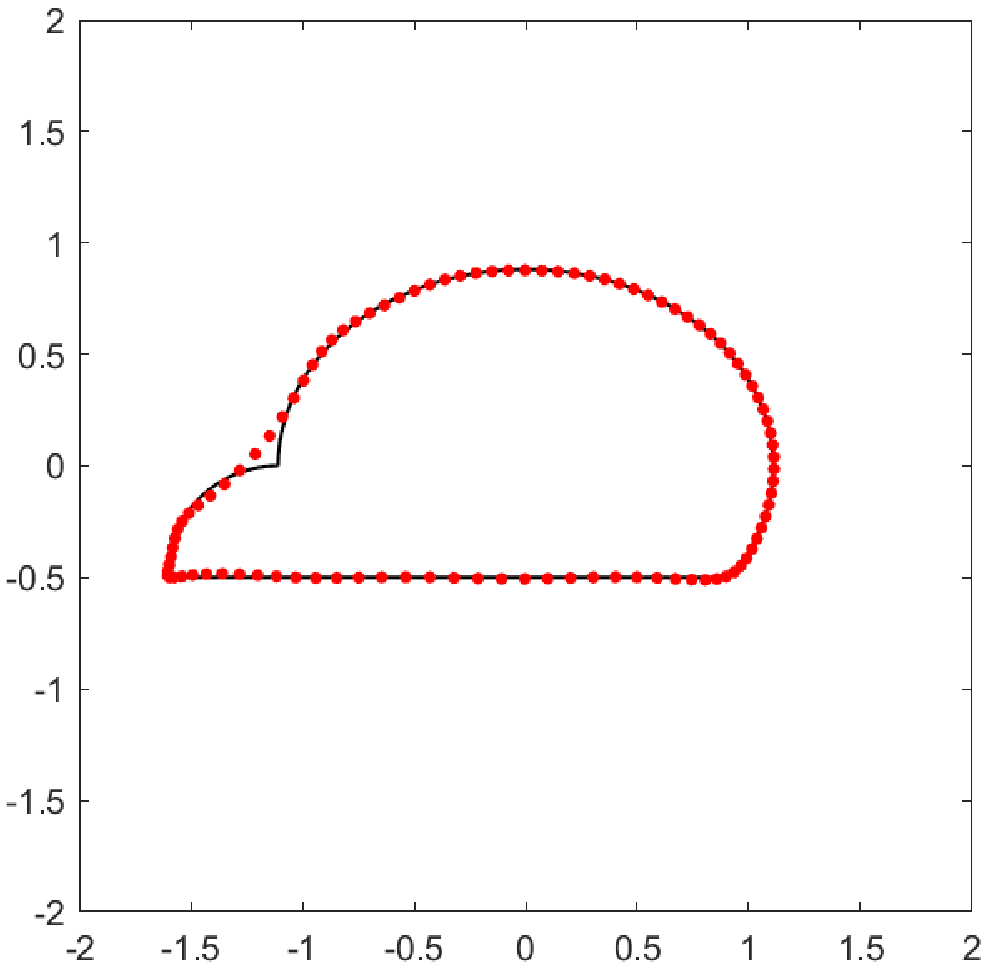}}
\subfloat[\label{Ex3-4}Ord$=20$,\newline $\lambda^{\mbox{rec}}\approx-1.5039$]{\includegraphics[width=0.24\textwidth]{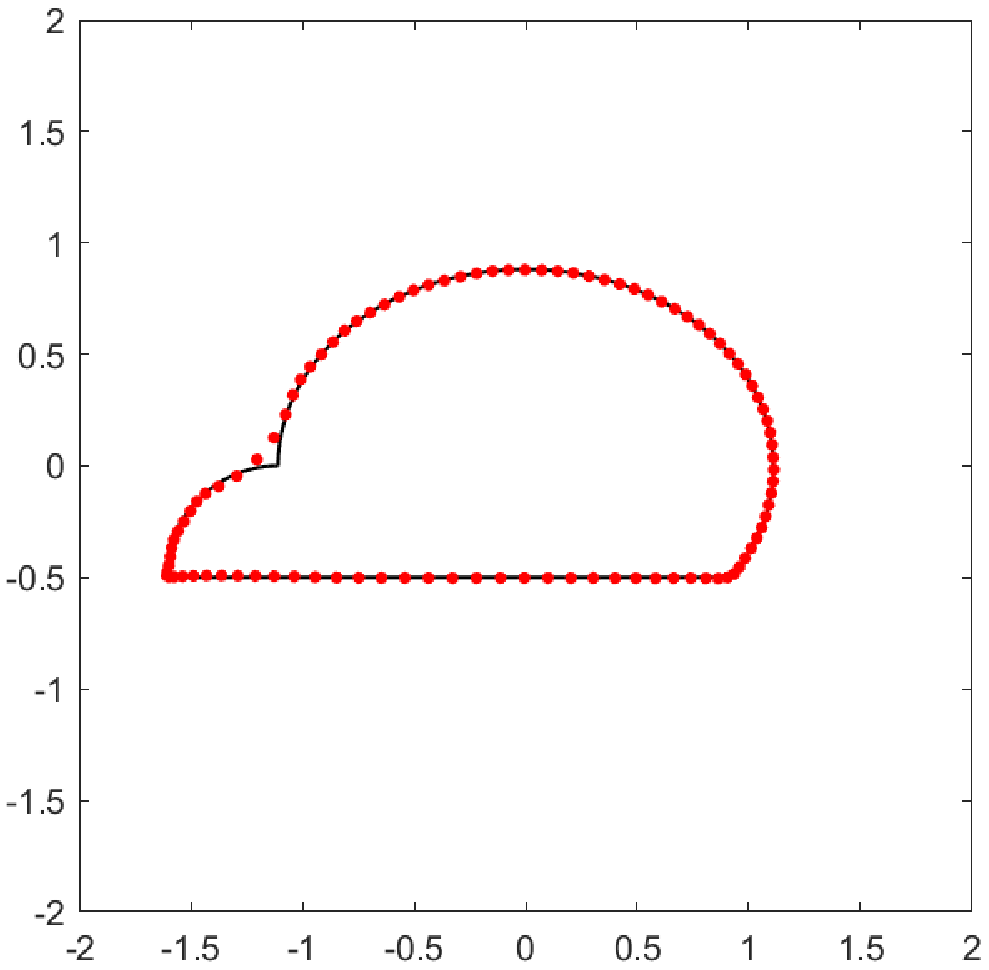}}
%\subfloat[Ord $25$]{\includegraphics[width=0.19\textwidth]{EX3-10}\label{Ex3-5}}
\caption{Recovery of the cap-shaped domain with $\sigma = 0.5$ (i.e., $\lambda = -1.5$) from the GPTs of orders up to $\textrm{Ord}=2,5,10,20$. The proposed reconstruction scheme works well for this inclusion with corners.}
\label{Ex3-result}
\end{figure}

\begin{figure}[t!]
	\captionsetup[subfloat]{margin=8pt,format=hang,singlelinecheck=false}
\centering
\subfloat[\label{Ex4-1}Ord$=2$,\newline $\lambda^{\mbox{rec}}\approx1.0028$]{\includegraphics[width=0.24\textwidth]{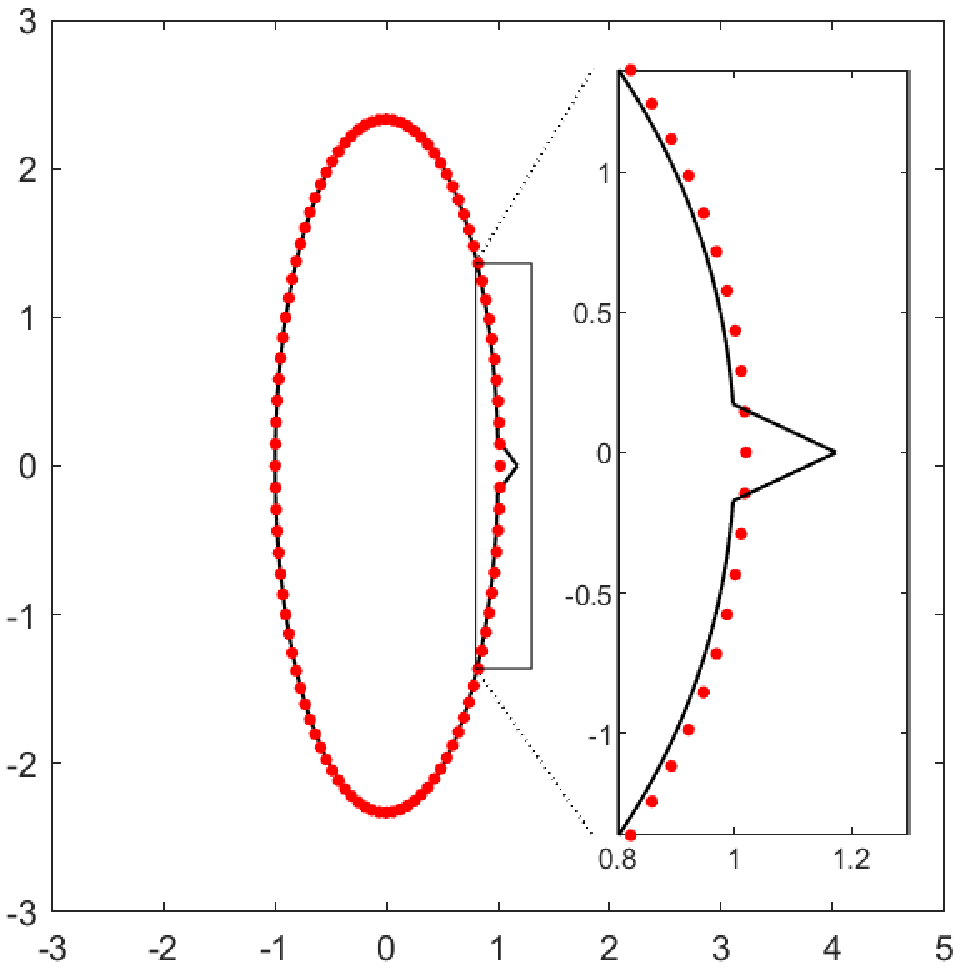}}
\subfloat[\label{Ex4-2}Ord$=10$,\newline $\lambda^{\mbox{rec}}\approx1.0019$]{\includegraphics[width=0.24\textwidth]{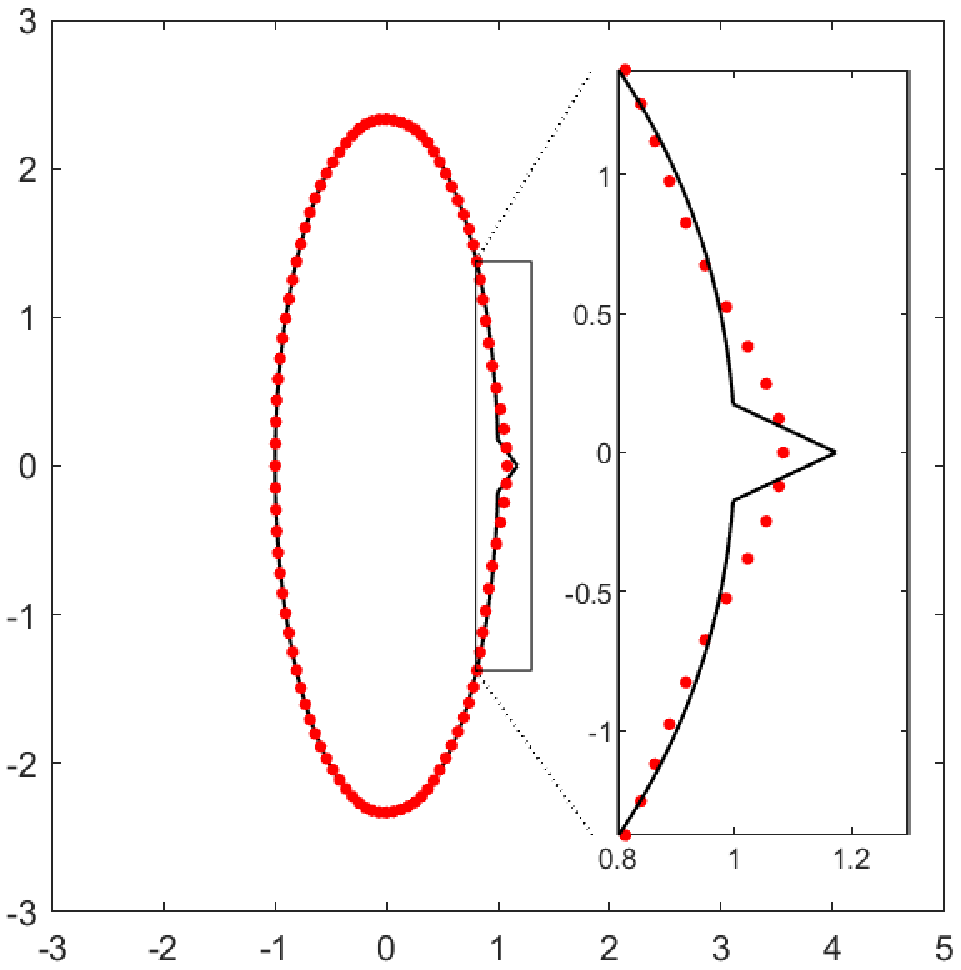}}
\subfloat[\label{Ex4-3}Ord$=15$,\newline $\lambda^{\mbox{rec}}\approx1.0012$]{\includegraphics[width=0.24\textwidth]{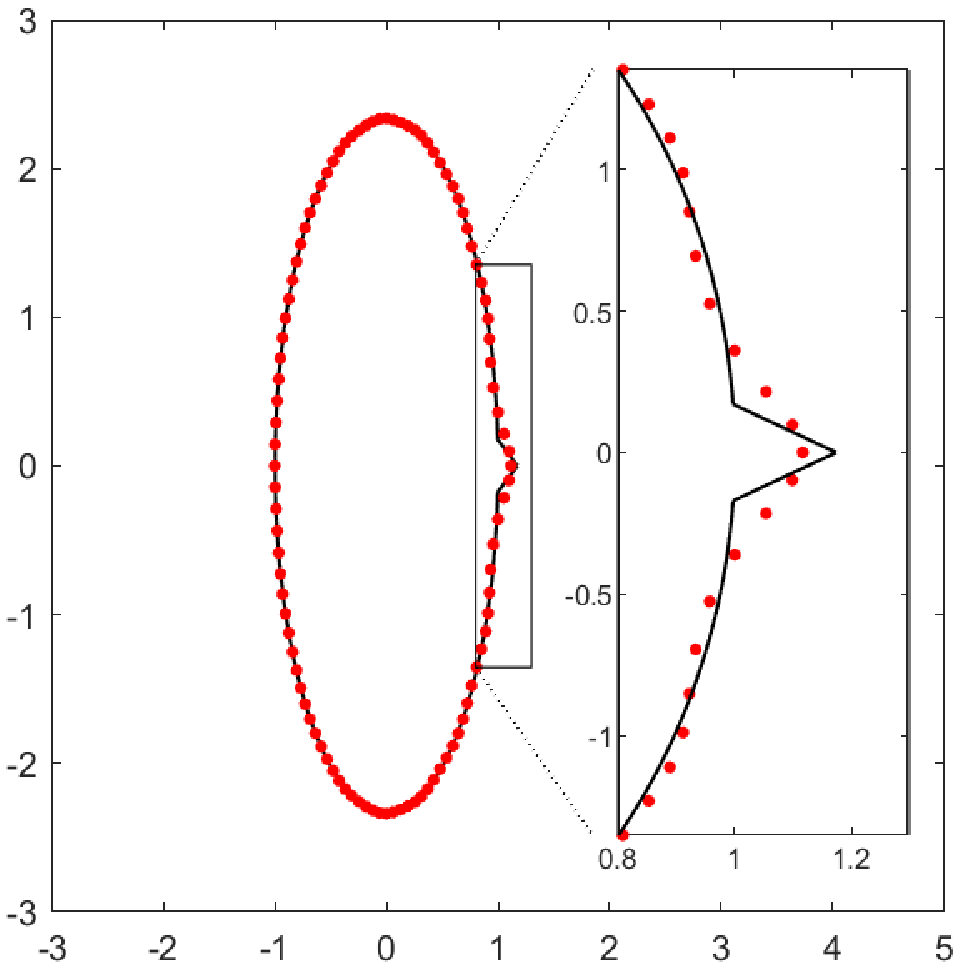}}
\subfloat[\label{Ex4-4}Ord$=25$,\newline $\lambda^{\mbox{rec}}\approx1.0003$]{\includegraphics[width=0.24\textwidth]{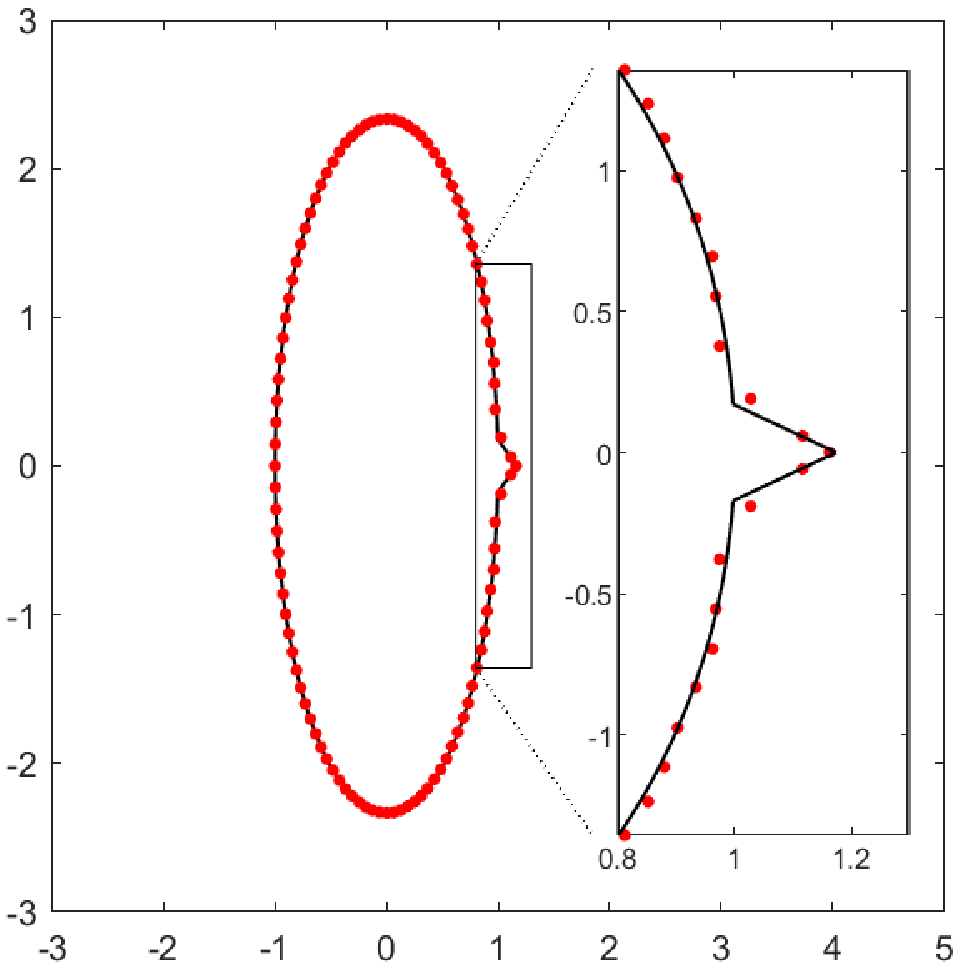}}
\caption{
Recovery of the perturbed ellipse with $\sigma = 3$ (i.e., $\lambda = 1$) from the GPTs of orders up to $\textrm{Ord}=2,10,15,25$. The image in a rectangle near the tiny corner is magnified. The proposed reconstruction scheme recovers the corner with $\textrm{Ord}=25$, while the corner is smoothened when $\textrm{Ord}$ is not large enough.}
\label{Ex4-result}
\end{figure}

%%%%%%%%%%%%%%%%%

%%%%%%%%%%%%%%%%%%%%%%%%%%%%%%%%%%%%%%%%%%%%%%%%%%%%%%%%%%%%%%%%%%%%%%%%%%%%%%%%%%%%%%%%%%%%%%%%%%%

\section{Conclusion}
We developed an analytical method of recovering a planar conductivity inclusion from exterior measurements, where the inclusion is assumed to be a simply connected domain and to be isotropic and homogeneous with arbitrary constant conductivity.  Based on the concept of the FPTs, we established matrix factorizations for the GPTs that hold for an inclusion with arbitrary constant conductivity,  where the inclusion is either a smooth domain or a star-shaped Lipschitz domain. The matrix factorizations lead us to an inversion formula for a conductivity inclusion. 
It would be of interest if the proposed inversion scheme--Theorems \ref{conformalrecovery} and \ref{thm:recovery:Lip}--could be extended to a general Lipschitz inclusion not assuming the star-shaped condition. We expect this generalization to be possible considering the shape monotonicity of the GPTs and the smooth approximation for a Lipschitz domain.

\end{document}